\documentclass[12pt]{amsart}
\usepackage{geometry}
\usepackage{mathtools}

\usepackage{amssymb,amsmath,amsthm}
\usepackage[english]{babel}
\usepackage{hyperref}

\newcommand{\hght}{\operatorname{ht}}
\newcommand{\length}{\operatorname{\lambda}}

\newcommand{\Spec}[1]{\operatorname {Spec(#1)}}
\newcommand{\red}[1]{#1_{\operatorname {red}}}
\newcommand{\mf}[1]{\mathfrak #1}

\DeclareMathOperator{\eh}{e}
\DeclareMathOperator{\ehk}{e_{HK}}
\DeclareMathOperator{\hkf}{\lambda_q}

\DeclareMathOperator{\Minh}{Minh}

\DeclareMathOperator{\Max}{Max}

\DeclareMathOperator{\Min}{Min}
\DeclareMathOperator{\Ass}{Ass}

\renewcommand{\frq}[1]{{#1}^{[q]}}
\renewcommand{\hat}{\widehat}

\newtheorem{theorem}{Theorem}
\newtheorem{lemma}[theorem]{Lemma}
\newtheorem{proposition}[theorem]{Proposition}
\newtheorem{corollary}[theorem]{Corollary}
\theoremstyle{definition}
\newtheorem{definition}[theorem]{Definition}

\theoremstyle{remark}
\newtheorem{remark}[theorem]{Remark}

\newtheorem{question}[theorem]{Question}

\begin{document}

\title[Equimultiplicity in Hilbert-Kunz theory]
{Equimultiplicity in Hilbert-Kunz theory}

\author{Ilya Smirnov}
\address{Department of Mathematics\\
University of Michigan\\
Ann Arbor, MI 48109, USA}
\email{ismirnov@umich.edu}

\dedicatory{Dedicated to Craig Huneke on the occasion of his 65th Birthday.}

\date{\today}

\begin{abstract}
We study further the properties of Hilbert-Kunz multiplicity as a measure of singularity.
This paper develops a theory of equimultiplicity for Hilbert-Kunz multiplicity and applies it
to study the behavior of Hilbert--Kunz multiplicity on the Brenner--Monsky hypersurface.
A number of applications follows, in particular 
we show that Hilbert--Kunz multiplicity attains infinitely many values
and that equimultiple strata may not be locally closed.
\end{abstract}


\maketitle

\section*{Introduction}

Hilbert--Samuel multiplicity is a classical invariant of a local ring 
that generalizes the notion of the multiplicity of a curve at a point.
The multiplicity may be regarded as a measure of singularity, where the lowest possible value, $1$,
corresponds to a smooth point.
In study of singularities, we are naturally led to study equimultiple points, 
i.e., a point such that the Hilbert--Samuel multiplicity is constant on 
the subvariety defined as the closure of the point. 
For example, it can be considered as the weakest form of equisingularity,
where we would say that two points are equally singular if the multiplicities are equal. 
This notion has been studied extensively, 
partially due to its appearance in Hironaka's work on the resolution of singularities in characteristic zero.

In $1983$ Monsky defined a new version of multiplicity, specific for positive characteristic. 
It was called Hilbert--Kunz multiplicity, in honor of Ernst Kunz who in $1969$ initiated the study of 
positive characteristic numerical invariants (\cite{Kunz1, Kunz2}). 
If we use $\length$ to denote the length of an artinian module, 
the Hilbert--Kunz multiplicity of a local ring $(R, \mf m)$ of characteristic $p$ 
can be defined by the limit
\[
\ehk (R) = \lim_{n \to \infty}\frac{1}{p^{n\dim R}} \length(R/\mf m^{[p^n]}),
\]
where we denoted the ideal of $p^n$-powers, $(x^{p^n} \mid x \in \mf m)$, by $\mf m^{[p^n]}$.
This invariant is still largely mysterious, and there is a lot to understand.
The purpose of this work is to build an equimultiplicity theory for Hilbert--Kunz multiplicity 
and to see what properties of the usual multiplicity are preserved.

As a tool to study singularities, Hilbert--Kunz multiplicity shares some properties with the classical multiplicity, 
but often has a more complicated behavior. While the Hilbert--Samuel multiplicity of a local ring is always an integer,
the Hilbert--Kunz multiplicity is not. 
However, this also allows us to think about Hilbert--Kunz multiplicity as a finer invariant.
Just like the usual multiplicity, it detects regular rings (Watanabe and Yoshida, \cite{WatanabeYoshida}); 
but it is also meaningful to consider rings of Hilbert--Kunz multiplicity 
sufficiently close to $1$, and expect that their singularities should get better.
Blickle and Enescu (\cite{BlickleEnescu} and its improvement \cite{AberbachEnescu} by Aberbach and Enescu)
showed that rings with sufficiently small Hilbert--Kunz multiplicity are Gorenstein and F-regular.

These results show that Hilbert--Kunz multiplicity is a measure of singularity,
and the next natural step is to study its geometric properties, with a view towards 
a possible use for the resolution of singularities. 

The first step was done by Kunz who showed in \cite[Proposition~3.3]{Kunz2}
that Hilbert--Kunz multiplicity satisfies the localization property:
$\ehk(\mf p) \leq \ehk(\mf m)$ for any prime ideal $\mf p$ in a local equidimensional catenary ring $(R, \mf m)$.
The author showed that Hilbert--Kunz multiplicity is upper semi-continuous (\cite{Smirnov}),
which implies that the maximum value locus is a subvariety.
Following Bierstone and Milman (\cite{BierstoneMilman}),
we may ask whether Hilbert--Kunz multiplicity satisfies the stronger condition of Zariski-semicontinuity; for example,
this would show that Hilbert--Kunz multiplicity on a fixed variety should attain only finitely many values.
As observed in \cite[Lemma~3.10]{BierstoneMilman}, Zariski-semicontinuity can be restated as openness 
for all $a$ of the sets
\[
X_{\leq a} = \{\mf p \in \Spec R \mid \ehk (\mf p) \leq a\}.
\] 
After using Nagata's criterion for openness,
this will require us to prove that for some $s \notin \mf p$ 
we can make the Hilbert--Kunz multiplicity to be constant on $V(\mf p) \cap D(s)$, or, as it fits to say,
make $\mf p$ {\it equimultiple} by inverting an element. 
In order to do so, we first need to understand the equimultiplicity condition
and the goal of this paper is to build such understanding.

\subsection{Overview of the results}
Before proceeding to our results, let us quickly review the classical picture.
The analytic spread of an ideal $I$, $\ell (I)$, is defined as the Krull dimension of the fiber cone
\[F(I) = R/\mf m \oplus I/\mf mI \oplus I^2/\mf mI^2 \oplus \cdots.\]
The analytic spread plays an important role in the classical Hilbert--Samuel theory, 
because it determines the size of a minimal reduction,
i.e., the number of general elements in $I$ needed to generate $I$ up to integral closure.
Let us use $\Minh (I)$ to denote the set of all prime ideals $P \supseteq I$ such that $\dim R/P = \dim R/I$, and
recall the classical characterization of equimultiple ideals (\cite{Rees, Lipman}).
\begin{theorem}\label{HS equimultiple}
Let $(R, \mf m, k)$ be a formally equidimensional local ring.
For an ideal $I$ the following conditions are equivalent:
\begin{enumerate}
\item[(a)] $\ell (I) = \hght (I)$,
\item[(b)] for every (equivalently, some) parameter ideal $J$ modulo $I$ 
\[\eh (I + J) = \sum_{P \in \Minh (I)} \eh(J, R/P) \eh(I, R_P),\] 
\item[(c)] if $k$ is infinite, there is a system of parameter $J = (x_1, \ldots, x_r)$ 
modulo $I$ which is a part of a system of parameters in $R$, and such that
\begin{itemize}
\item if $r < \dim R$ then $\sum_{P \in \Minh (I)} \eh(J, R/P) \eh(IR_P) = \eh(I R/J)$
\item if $r = \dim R$ then $\sum_{P \in \Minh (I)} \eh(J, R/P) \eh(IR_P) = \eh (J)$,
\end{itemize}
\item[(d)] for every (equivalently, some) system of parameters $(x_1, \ldots, x_r)$ modulo $I$, 
all $0 \leq i < r$, and all $n$ 
\[
x_{i + 1} \text{ is regular modulo } \overline{(I, x_1, \ldots, x_i)^n}.
\]	
\end{enumerate}
\end{theorem}
An ideal that satisfies any of these equivalent condition is called {\it equimultiple}.
This name originates from the special case when $I = \mf p$ is a prime ideal such that $R/\mf p$ is regular. 
In this case the condition (b) means that $\eh (\mf m) = \eh (\mf p)$.

In this paper we develop a theory of {\it Hilbert--Kunz equimultiple} ideals.
Again, if $I = \mf p$ is prime and $R/\mf p$ is regular, then $\mf p$ is Hilbert--Kunz equimultiple
if and only if $\ehk (\mf m) = \ehk (\mf p)$. 
The obtained characterization is very similar: we replace powers and integral closure with Frobenius powers and tight closure.
The most important difference is that an analogue of part (a) is missing
because we are not able to find a satisfactory replacement for analytic spread (see Question~\ref{Spread} and 
the surrounding discussion). 
Our findings are collected in the following theorem 
(Theorem~\ref{HK equimultiple}, Corollaries~\ref{excellent capture} and \ref{modulo sop}).

\begin{theorem}\label{HK intro}
Let $(R, \mf m)$ be a local ring of positive characteristic $p > 0$.
Furthermore suppose that $R$ is either formally unmixed with a locally stable test element $c$
or is excellent and equidimensional (e.g., a complete domain).
For an ideal $I$ the following conditions are equivalent:
\begin{enumerate}
\item[(b)] for every (equivalently, some) parameter ideal $J$ modulo $I$ 
\[\ehk (I + J) = \sum_{P \in \Minh (I)} \ehk(J, R/P) \ehk(I, R_P),\] 
\item[(c)] there is a system of parameter $J = (x_1, \ldots, x_r)$ 
modulo $I$ which is a part of a system of parameters in $R$, and such that
\begin{itemize}
\item if $r < \dim R$ then $\sum_{P \in \Minh (I)} \ehk(J, R/P) \ehk(IR_P) = \ehk(I R/J)$
\item if $r = \dim R$ then $\sum_{P \in \Minh (I)} \ehk(J, R/P) \ehk(IR_P) = \ehk (J)$,
\end{itemize}
\item[(d)] for every (equivalently, some) system of parameters $(x_1, \ldots, x_r)$ modulo $I$, 
all $0 \leq i < r$, and all $q = p^e$ 
\[
x_{i + 1} \text{ is regular modulo } ((I, x_1, \ldots, x_i)^{[q]})^*.
\]	
\end{enumerate}
\end{theorem}

The main consequence of this characterization is a surprising negative answer (Corollary~\ref{not locally constant}): 
Hilbert--Kunz multiplicity is not Zariski-semicontinuous in the sense of \cite{BierstoneMilman}.

\begin{corollary}
Let $R = F[x, y, z, t]/(z^4 +xyz^2 + (x^3 + y^3)z + tx^2y^2)$, where $F$ is the algebraic closure of $\mathbb Z/2\mathbb Z$.
Then the set
\[
X_{\leq 3} = \{\mf p \in \Spec R \mid \ehk (\mf p) \leq 3\}
\]
is not open.
\end{corollary}

This readily implies that Hilbert--Kunz multiplicity attains infinitely many values 
(Corollary~\ref{inf many values}) on the Brenner--Monsky hypersurface $z^4 +xyz^2 + (x^3 + y^3)z + tx^2y^2 = 0$.
This hypersurface was previously used by Brenner and Monsky to provide a counter-example to the notorious 
localization problem in tight closure theory (\cite{BrennerMonsky}). 

In order to understand the Brenner--Monsky counter-example better and compare it to the results of \cite{HochsterHuneke3},
Dinh has studied the associated primes of the Frobenius powers of the prime ideal $P = (x,y,z)$ 
defining the singular locus of the aforementioned hypersurface.
Using the calculations that Monsky made to obtain Theorem~\ref{BM ehk}, 
Dinh (\cite{Dinh}) proved that $\bigcup_q \Ass (\frq{P})$ is infinite. 
However, he was only able to show that the maximal ideals corresponding to the
irreducible factors of $1 + t + t^2 + \ldots + t^q$ appear as associated primes,
while our methods immediately give all associated primes of the Frobenius powers and their tight closures.

\begin{proposition}
In the Brenner--Monsky example, 
\[\bigcup_q \Ass (\frq{P})^* = \bigcup_q \Ass (\frq{P}) = \Spec {R/P}.\] 
In particular, the set is infinite.
\end{proposition}

Our applications demonstrate that in Hilbert--Kunz theory equimultiplicity is a far more restrictive condition
than equimultiplicity in Hilbert--Samuel theory.
This fits well with our understanding of Hilbert--Kunz multiplicity as a finer invariant; 
when there are many more possible values it is less likely for the values at two points to coincide.

The presented applications were obtained in the Brenner--Monsky hypersurface
because we use Monsky's computations in \cite{MonskyQP}.
Potentially the theory can applied to any other example
where the Hilbert--Kunz multiplicity was computed in a family (e.g., \cite{MonskyQL});
we chose the simplest and the best understood family.

It is also worth to highlight a very surprising corollary (Corollary~\ref{weakly f-regular}), obtained by
combining Theorem~\ref{HK intro} with Proposition~\ref{HK functions}, 
that gives a criterion for equality of the entire Hilbert--Kunz functions of $\mf p$ and $\mf m$.
\begin{corollary}
Let $(R, \mf m)$ be a weakly F-regular excellent local domain and let $\mf p$ be a prime ideal such that $R/\mf p$ is regular.
Then $\ehk (\mf m) = \ehk(\mf p)$ if and only if the entire
 Hilbert--Kunz functions (not just the limits!) of $R$ and $R_\mf p$ coincide. 
\end{corollary}

Last, I want to remark that the presented theory has recently
found an application in Linquan Ma's breakthrough on Lech's conjecture (\cite{Ma}).
As a key ingredient of his proof, Ma showed that perfect ideals are Hilbert--Kunz equimultiple.

\subsection{The structure of the paper}
The absence of reductions and analytic spread required us to develop new tools.
The uniform convergence methods provide a crucial tool for our proof, Corollary~\ref{ehk limits}, 
and we use Tucker's ideas from \cite{Tucker} to develop Section~\ref{HK equi convergence}.
Corollary~\ref{ehk limits} is interesting in its own right and might be useful in a variety of situations
(such as Ma's recent work in \cite{Ma}).

We use this tool in Section~\ref{HK equi general}, where we build the bulk of the Hilbert--Kunz equimultiplicity theory,
resulting in Theorem~\ref{HK equimultiple} and its corollaries. 
Perhaps the most challenging part of the proof is an analogue of the equivalence between (b) and (d) of 
Theorem~\ref{HS equimultiple}, which we achieve in Theorems~\ref{coloninc2}, \ref{capturing equivalence}.
Later in the section we study the general behavior of Hilbert--Kunz equimultiplicity under ring operations in order to relax some of 
the assumptions of the aforementioned theorems.
We also investigate the equivalence between ``some'' and ``every''
in that analogue of part (d) under milder assumptions (Corollary~\ref{equivalence}).
In the end of the section, we study the behavior of Hilbert--Kunz equimultiplicity under localization and specialization:
Proposition~\ref{equi localizes} shows that the property is preserved under localization, and Corollary~\ref{HK specialization}
provides a Bertini-type result. These results are used to establish, in Corollary~\ref{modulo sop}, 
the condition (c) of Theorem~\ref{HK intro}.

In Section~\ref{Applications}, we will use the machinery we have constructed
to study the Brenner--Monsky hypersurface and obtain a number of interesting results.

As a first step, in Corollary~\ref{global char} we establish equivalence  
between equimultiplicity of prime ideal $\mf p$ of dimension one and the $\mf p$-primary property of the tight closures 
$(\frq{\mf p})^*$. Then we use the previously mentioned results on specialization 
(Corollaries~\ref{equi specialization}, \ref{equi prime modulo})  in Proposition~\ref{HK counterexample}, where 
we derive from Monsky's computations in \cite{MonskyQP} 
that any point on the curve defining the singular locus has Hilbert--Kunz multiplicity greater than the generic point.
This implies that the Hilbert--Kunz multiplicity attains infinitely many values on the curve.
Hence, the stratification by Hilbert--Kunz equimultiple strata need not be finite (Corollary~\ref{inf many values}).
Also, in Corollary~\ref{not locally constant}, we will show that
an individual stratum need not be locally closed
and the set $X_{\leq a}$ need not be open.

As a second application of our methods, in Proposition~\ref{infinite ass} we easily compute 
the set of all associated primes of the Frobenius powers (or tight closures of the Frobenius powers) 
for the defining ideal of the previously mentioned curve. 
This improves a result of Dinh (\cite{Dinh}) who showed that the set is infinite.

In the end of the paper, we give a list of open problems.

\section{Preliminaries}

\numberwithin{theorem}{section}
\numberwithin{equation}{section}

In this paper we study commutative Noetherian rings, typically of characteristic $p > 0$. 
For convenience, we use notation $q = p^e$ where $e \in \mathbb N$ may vary.
For an ideal $I$ of $R$ let $\frq{I}$ be the ideal generated by the $q$th powers of the elements of $I$.

\subsection {Tight closure}

Let us briefly review some results from tight closure theory that will be needed for our work. 
For proofs and a more detailed exposition we refer the reader to \cite{HochsterHuneke1}.

\begin{definition}
Let $R$ be a ring and let $I$ be an ideal of $R$. 
The tight closure $I^*$ of $I$ consists of all elements $x$ of $R$ such that there exists 
a fixed element $c$ such that $c$ is not contained in any minimal prime of $R$ and 
\[
cx^q \in \frq {I}
\] 
for all sufficiently large $q$.
\end{definition}

\begin{definition}
Let $R$ be a Noetherian ring of characteristic $p > 0$. 
We say that an element $c$ is a {\it test element} if $c$ is not contained in any minimal prime and
for every ideal $I$ and element $x \in R$ 
$x \in I^*$ if and only if $cx^q \in \frq{I}$ for all $q$.

Furthermore, we say that $c$ is a {\it locally stable test element} 
if the image of $c$ in $R_P$ is a test element for any prime $P$.
\end{definition}

Test elements make tight closure a very powerful tool and have had a great use.
After a tremendous amount of work in \cite[Theorem~6.1]{HochsterHuneke2}, 
Hochster and Huneke obtained the following existence theorem.

\begin{theorem}\label{test exist}
Let $R$ be reduced and either F-finite or essentially finite type over an excellent local ring.
Then $R$ has a locally stable test element. 
\end{theorem}

Test elements provide us with a very useful lemma.

\begin{lemma}\label{intersection}
Let $(R, \mf m)$ be a local ring of positive characteristic with a test element $c$. 
If $I$ and $J$ are two proper ideals then 
\[
I^* = \cap_q (I, \frq{J})^*.
\]
\end{lemma}
\begin{proof}
This easily follows from the definition: if $x$ belongs to the intersection, then for all $q$
\[
cx^{q'} \in \cap_q (I, \frq{J})^{[q']}
\]
and we are done by Krull's intersection theorem.
\end{proof}

\subsection{Tight closure and Hilbert--Kunz multiplicities}

For our work we will need to extend the definition given in the Introduction. 
Its existence was proved by Monsky (\cite{Monsky}).

\begin{definition}
Let $(R, \mf m)$ be a local ring of characteristic $p > 0$, $I$ be an $\mf m$-primary ideal,
and $M$ be a finite $R$-module. 
The limit  
\[
\ehk(I, M) := \lim_{q \to \infty} \frac{1}{q^{\dim R}} \length(M/\frq{I}M),
\]
is called the Hilbert--Kunz multiplicity of $M$ with respect to $I$.
\end{definition}

If $I= \mf m$, we will call it the {\it Hilbert--Kunz multiplicity} of $M$ and denote it by $\ehk(M)$. 
For a prime ideal $\mf p$, it is convenient to denote $\ehk (\mf pR_\mf p)$ by $\ehk (\mf p)$.
This naturally makes Hilbert--Kunz multiplicity to be a function on the spectrum and
the goal of this paper is to understand when $\ehk(\mf m) = \ehk (\mf p)$.

Now we discuss a very useful connection between Hilbert--Kunz multiplicity and tight closure.
For two ideals $I \subset J$, it is easy to see that $\ehk (I) \geq \ehk(J)$, 
but equality may hold despite that the two ideals are distinct.
The following theorem, due to Hochster and Huneke (\cite[Theorem~8.17]{HochsterHuneke1}), 
describes when does it happen.

Recall that a local ring is formally unmixed if $\Ass (0\widehat{R}) = \Minh (0\widehat R)$, or, in other words, 
$\dim \widehat{R}/P = \dim R$ for every associated prime of $\widehat {R}$.
For example, a complete domain is formally unmixed.

\begin{theorem}\label{HH TC}
Let $(R, \mf m)$ be a formally unmixed local ring and  $I \subseteq J$ are ideals.
Then $J \subseteq I^*$ if and only if $\ehk(I) = \ehk(J)$.
\end{theorem}
 
The following lemma shows that the Hilbert--Kunz multiplicity of $I$ can be computed using the filtration $(\frq{I})^*$;
it can be thought of as a generalization of the ``only if'' direction. 
We are interested in this filtration, since it is often more useful then the usual filtration $\frq{I}$.

\begin{remark}
If $R$ has a test element $c$, then, by definition, $c\sqrt{0} = c 0^* = 0$.
Since $c$ does not belong to any minimal prime, it follows that $R_P$ is reduced for any minimal prime $P$ of $R$.
Therefore the dimension of the nilradical of $R$ is less than $\dim R$.
Thus $\ehk(I, M) = \ehk(I, \red{R} \otimes_R M)$ for any $\mf m$-primary ideal $I$ and module $M$.
\end{remark}

\begin{lemma}\label{tcinvariance}
Let $(R, \mf m)$ be a local ring of characteristic $p > 0$, $I$ be an $\mf m$-primary ideal, 
and $M$ be a finitely generated $R$-module. If $R$ has a test element $c$, then
\[
\lim\limits_{q \to \infty} \frac{1}{q^d} \length(M/(\frq{I})^*M) = \ehk (I, \red{R} \otimes_R M) = \ehk (I, M).
\]
\end{lemma}
\begin{proof}
First, by \cite[Proposition~4.1(j), Proposition~6.1(d)]{HochsterHuneke1}, $c$ is still a test element in $\red{R}$ and
$R/(\frq{I})^* \cong \red{R}/(\frq{I}\red{R})^*$.
Thus, using the previous remark, we assume that $R$ is reduced. 

Now, consider the exact sequence 
\[
R \xrightarrow{c} R \to R/(c) \to 0.
\]
Since $c(\frq{I})^* \subseteq \frq{I}$, we obtain that the sequence
\[
R/(\frq{I})^* \otimes_R M \xrightarrow{c} R/\frq{I} \otimes_R M \to R/(c, \frq{I}) \otimes_R M \to 0.
\]
is still exact. 
Together with the inclusion $\frq{I} \subseteq (\frq{I})^*$ this shows that
\[
\length (M/(\frq{I})^*M) \leq \length (M/\frq{I}M) \leq \length (M/(\frq{I})^*M) + \length (M/(c, \frq{I})M).
\]
Since $c$ is not contained in any minimal prime, $\dim M/cM \leq \dim R/(c) < \dim R$.
Therefore $\ehk(I, M) = \lim\limits_{q \to \infty} \frac{1}{q^d} \length(M/(\frq{I})^*M)$.
\end{proof}

We will need the following corollary to deal with localization of tight closure.
Unfortunately, tight closure does not commute with localization, but there is still the inclusion
$I^* R_{\mf p} \subseteq (IR_\mf p)^*$, where the first closure is taken in $R$ and the second is in $R_\mf p$.
Thus, the corollary allows us to compute the Hilbert--Kunz multiplicity of $R_\mf p$ by taking 
the filtration $(\frq{\mf p})^*R_\mf p$.

\begin{corollary}\label{tc sequence limit}
Let $(R, \mf m)$ be a local ring of characteristic $p > 0$ with a test element $c$.
Let $I$ be an $\mf m$-primary ideal and $I_q$ be a sequence of ideals such that 
$\frq{I} \subseteq I_q \subseteq (\frq{I})^*$. Then 
\[\lim\limits_{q \to \infty} \frac{1}{q^d} \length(R/I_q) = \ehk (I).\]
\end{corollary}
\begin{proof}
The claim follows from Lemma~\ref{tcinvariance}, since the inclusions 
$\frq{I} \subseteq I_q \subseteq (\frq{I})^*$ give that
\[
\ehk(I) \geq \lim\limits_{q \to \infty} \frac{1}{q^d} \length(R/I_q) 
\geq \lim\limits_{q \to \infty} \frac{1}{q^d} \length(R/(\frq{I})^*) = \ehk (I).
\]
\end{proof}

More importantly, there is a partial converse to this equality improving further Theorem~\ref{HH TC}.
It will be very useful later, as it provides us with a way to detect when a filtration is in tight closure.

\begin{lemma}\label{getintc}
Let $(R, \mf m)$ be an formally unmixed local ring of characteristic $p > 0$ and $I$ be an $\mf m$-primary ideal.
If $I_q$ is a sequence of ideals such that $\frq{I} \subseteq I_q$, $I_q^{[q']} \subseteq I_{qq'}$ for all $q, q'$, and 
\[\lim_{q \to \infty} \frac{\length{(R/I_q)}}{q^d} = \ehk(I),\]
then $I_q \subseteq (\frq{I})^*$ for all $q$.
\end{lemma}
\begin{proof}
By the assumptions on $I_q$, 
\[I^{[qq']} \subseteq I_q^{[q']} \subseteq I_{qq'}.\]
Therefore,
\[q^d \ehk (I) = \ehk (\frq {I}) \geq \ehk (I_q) \geq \lim_{q' \to \infty} \frac{\length{(R/I_{qq'})}}{(q')^d} = q^d\ehk(I),\]
so $\ehk (\frq{I}) = \ehk(I_q)$.
Hence, $I_q \subseteq (\frq{I})^*$  by Theorem~\ref{HH TC}.
\end{proof}

The following lemma helps us to understand what it means for an element to be regular modulo 
tight closures of consecutive powers. Recall that for an ideal $I$ and an element $x \notin I$
we denote $I : x^\infty = \bigcup_n I : x^n$.

\begin{lemma}\label{tc nzd}
Let $(R, \mf m)$ be a local ring of characteristic $p > 0$, $I$ an ideal, and $x$ an element.
Suppose that $R$ has a test element $c$, then the following are equivalent:
\begin{enumerate}
\item[(a)] $x$ is not a zero divisor modulo $(\frq{I})^*$ for any $q$, 
\item[(b)] $\frq{I}: x^\infty \subseteq (\frq{I})^*$ for all $q$,
\item[(c)] for all $q$ there are ideals $I_q$ such that $x$ is not a zero divisor modulo $I_q$ and 
$\frq{I} \subseteq I_q \subseteq (\frq{I})^*$.
\end{enumerate}
\end{lemma}
\begin{proof}
$(a) \Rightarrow (b)$ since $\frq{I}: x^\infty \subseteq (\frq{I})^*: x^\infty$. $(b)\Rightarrow (c)$ 
by taking $I_q = \frq{I} : x^\infty$. 
For the last implication, we observe that if $ax \in (\frq{I})^*$ for some $q$, then $ca^{q'}x^{q'} \in I^{[qq']}$
for all $q' \gg 0$. But, since $x^{q'}$ is not a zero divisor modulo $I_{qq'}$,
\[
ca^{q'} \in I_{qq'} \subseteq (I^{[qq']})^*.
\]
Since $R$ has a test element, these equations imply that $a \in (\frq{I})^*$. 
\end{proof}

\begin{corollary}\label{tc primarity}
Let $(R, \mf m)$ be a local ring of positive characteristic and $\mf p$ be a prime ideal. 
If $R$ has a test element $c$, then the following are equivalent:
\begin{enumerate}
\item[(a)] $(\frq{\mf p})^*$ is $\mf p$-primary for any $q$,
\item[(b)] $\frq{\mf p}R_\mf p \cap R \subseteq (\frq{\mf p})^*$ for all $q$,
\item[(c)] for all $q$ there exist $\mf p$-primary ideals $I_q$ such that
$\frq{\mf p} \subseteq I_q \subseteq (\frq{\mf p})^*$.

\end{enumerate}
\end{corollary}
\begin{proof}
Clearly, $I_q = \frq{\mf p}R_\mf p \cap R $ yields $(b) \Rightarrow (c)$.

For an ideal $I$ such that $\sqrt{I} = \mf p$, we can characterize its $\mf p$-primary part as 
the smallest among the ideals containing $I$ 
and such that any element $x \notin \mf p$ is not a zero divisor modulo that ideal.
Thus, for $(c) \Rightarrow (a)$, we note that $\frq{\mf p}: x^\infty \subseteq I_q$ and use the lemma above.

For the last implication, we just note that $\frq{\mf p}R_\mf p \cap R \subseteq (\frq{\mf p})^*R_\mf p \cap R$.

\end{proof}

\subsection {Basic properties of multiplicities}

Let us overview some properties of multiplicities and Hilbert--Kunz multiplicities that 
will be used throughout the text. First, there is a very useful additivity formula (\cite{Lech}, \cite[(2.3)]{WatanabeYoshida}).

\begin{lemma}[Additivity formula]\label{ass formula}
Let $(R, \mf m)$ be a local ring of dimension $d$ and $I$ an arbitrary $\mf m$-primary ideal.
Then 
\[\eh (I) = \sum_{P \in \Minh R} \eh(IR/P) \length_{R_P} (R_P),\]
where the sum is taken over all primes $P$ such that $\dim R/P = \dim R$.

Similarly, if $R$ has positive characteristic, then
\[\ehk (I) = \sum_{P \in \Minh R} \ehk(IR/P) \length_{R_P} (R_P).\]
\end{lemma}

In this paper we will often work with parameter ideals. 
To make our notation less cumbersome, in the following 
for  a system of parameters  $x_1, \ldots, x_m \in R$ modulo $I$ we are going to write 
$\ehk(I, x_1, \ldots, x_m)$ instead of $\ehk((I, x_1, \ldots, x_m))$.

It is well-known that the multiplicity of a regular sequence can be computed as a colength.

\begin{proposition}\label{hs param}
Let $(R, \mf m)$ be a local ring of dimension $d$, $x_1, \ldots, x_r$ a system of parameters,
and $I$ an arbitrary $\mf m$-primary ideal. 
Then $\length \left (R/(x_1, \ldots, x_r) \right ) \geq \eh(x_1, \ldots, x_r)$. 

Moreover, $\length \left (R/ (x_1, \ldots, x_r) \right ) = \eh(x_1, \ldots, x_r)$ 
if and only if $x_1, \ldots, x_r$ form a regular sequence.
\end{proposition}

It was shown by Lech (\cite{Lech}) that for a parameter ideal $(x_1, \ldots, x_d)$, 
its multiplicity can be computed by the formula
\[
\eh \left(x_1, \ldots, x_r \right) = \lim_{\min (n_i) \to \infty} 
\frac{\length (R/(x_1^{n_1}, \ldots, x_d^{n_d}))}{n_1 \cdots n_d}.
\]
This easily implies two very useful corollaries.

\begin{corollary}\label{ehk parameter}
Let $(R, \mf m)$ be a local ring of characteristic $p > 0$ and let $x_1, \ldots, x_d$ be a system of parameters.
Then 
\[\ehk(x_1, \ldots, x_d) = \eh(x_1, \ldots, x_d).\]
\end{corollary}

\begin{corollary}\label{multiplicity of param power}
Let $(R, \mf m)$ be a local ring and let $x_1, \ldots, x_d$ be a system of parameters.
Then for any vector $(n_1, \ldots, n_d) \in \mathbb{N}^d$
\[
\eh (x_1^{n_1},\ldots, x_d^{n_d}) = n_1\cdots n_d \eh (x_1, \ldots, x_d).
\]
\end{corollary}

 Lech also used this observation to obtain the following associativity formula for parameter ideals (\cite[Theorem~1]{Lech}, \cite[(24.7), (34.5)]{Nagata}).

\begin{proposition}\label{param ass formula}
Let $(R, \mf m)$ be a local ring and let $x_1, \ldots, x_d$ be a system of parameters.
Then for any $0 \leq i \leq d$
\[
\eh(x_1,\ldots, x_d) = \sum_{P} \eh((x_{i+1}, \ldots, x_d), R/P) \eh((x_1, \ldots, x_i)R_P),
\]
where the sum is taken over all $P \in \Minh ((x_1, \ldots, x_i))$ such that $\dim R_P = i$.
In particular, if $R$ is formally equidimensional\footnote{The published version is missing this assumption, but this does not affect the results as we apply the proposition only to formally equidimensional rings.}, then the sum is taken over 
$P \in \Minh ((x_1, \ldots, x_i))$.
\end{proposition}



\section{Equimultiplicity for Hilbert--Kunz functions}\label{equimultiplicity}

First recall that the sequence  
\[q \to \hkf (\mf m) := \frac{1}{q^{\dim R}} \length(R/\frq{\mf m}R),
\]
is called the (normalized) Hilbert--Kunz function of a local ring $(R, \mf m)$.
As a first step, let us characterize for what prime ideals $\mf p$ 
the entire Hilbert--Kunz functions of $R_\mf p$ and $R$ coincide.
The presented argument is based on the proof of \cite[Theorem~3.3]{HunekeYao}.




\begin{proposition}\label{HK functions}
Let $(R, \mf m)$ be a local ring of characteristic $p > 0$ 
and $\mf p$ be a prime ideal of $R$ such that $R/\mf p$ is regular and $\hght \mf p  + \dim R/\mf p = \dim R$.
Then, for a fixed $q$, $\hkf (\mf m) = \hkf (\mf p)$ if and only if $R/\frq{\mf p}$ is Cohen-Macaulay.

Therefore the (normalized) Hilbert--Kunz functions of $\mf m$ and $\mf p$ coincide for all $q$
if and only if $R/\frq{\mf p}$ are Cohen-Macaulay for all $q$. 
\end{proposition}
\begin{proof}
Let $x_1, \ldots, x_m$ be a minimal system of generators of $\mf m$ modulo $\mf p$.
By the additivity formula and Corollary~\ref{multiplicity of param power},
\[
\eh(\frq{(x_1, \ldots, x_m)}, R/\frq{\mf p}) = \eh(\frq{(x_1, \ldots, x_m)}, R/\mf p)
\length_{R_{\mf p}}(R_{\mf p}/\frq{\mf p}R_{\mf p})=
q^{m} \length_{R_{\mf p}}(R_{\mf p}/\frq{\mf p}R_{\mf p}) 
.\]
Since $q^{m} \length_{R_{\mf p}}(R_{\mf p}/\frq{\mf p}R_{\mf p}) = q^{\dim R} \hkf(\mf p)$, 
we obtain by Proposition~\ref{hs param} that
\[
q^{\dim R}\hkf(\mf m) = \length(R/(\frq{\mf p}, x_1^q, \ldots, x_m^q)) \geq 
\eh(\frq{(x_1, \ldots, x_m)}, R/\frq{\mf p}) = q^{\dim R} \hkf(\mf p).
\] 
Thus, $\hkf(\mf m) = \hkf(\mf p)$  if and only if  
$\length(R/(\frq{\mf p}, x_1^q, \ldots, x_m^q)) = \eh(\frq{(x_1, \ldots, x_m)}, R/\frq{\mf p})$.
However, by Proposition~\ref{hs param}, 
the latter holds if and only if $R/\frq{\mf p}$ is Cohen-Macaulay.
\end{proof}

While the characterization is simple and natural, it is not clear 
whether we could force this condition by inverting an element. 
For a fixed $q$ this condition holds on an open set $D(s_q)$ because
the Cohen-Macaulay locus is open in an excellent ring.
However, we will need to intersect infinitely many open sets $D(s_q)$ to force it for all $q$.

This is why we need to go further and try to characterize equality of Hilbert--Kunz multiplicities.
It will be much harder to achieve, but 
having a better control over Hilbert--Kunz multiplicity, we will learn in Proposition~\ref{HK counterexample} that, in fact, 
the conditions of Proposition~\ref{HK functions} cannot be forced for all $q$ by inverting a single element.

Last, we remark that this characterization
can be used to give a different proof of the following classical theorem that was first proved by
Kunz in \cite[Corollary~3.4]{Kunz2} (and another simple proof is due to Shepherd-Barron in \cite{ShepherdBarron}).

\begin{theorem}\label{qupsemi}
If $R$ is an excellent locally equidimensional ring, then for any fixed $q$ 
the $q$th Hilbert--Kunz function $\hkf$ is locally constant on $\Spec R$.
\end{theorem}

\section{A uniform convergence result}\label{HK equi convergence}

Before proceeding to technicalities, let us sketch the ideas of the proof.
Over a local ring $(R, \mf m)$ we are going to prove uniform convergence (with respect to $q$) of the bisequence
\[
\frac{\length (M/(\frq{I} + J^{[qq']})M)}{q^{\dim R}q'^{\dim R/I}}
\]
where $I$ is an ideal, $M$ is a finitely generated module, and $J$ is an $\mf m$-primary ideal. 
This allows us to interchange the limits (with respect to $q$ and $q'$) of the bisequence\footnote{Neil Epstein informed me that a more general change of limits formula was independently obtained in his joint work with
Yongwei Yao.}.

Uniform convergence will be established by showing that the sequence is Cauchy with an appropriate estimate; 
to do the bookkeeping of estimates we follow Tucker's treatment in \cite{Tucker}. 
Tucker's proof can be viewed as a careful adaptation of the original proof by Monsky (\cite{Monsky}). 

We start with an upper bound for a function that we study.

\begin{lemma}\label{boundlemma}
Let $(R, \mf m)$ be a local ring and $I$ be an ideal.
Then there exists a constant $C$ such that
\[
\length (R/(\frq{I}+\mf m^{[qq']})) \leq C q'^{\dim R/I}q^{\dim R}
\]
for all $q, q'$.
\end{lemma}
\begin{proof}
Let $x_1, \ldots, x_h$ be elements of $R$ such that their images form a system of parameters in $R/I$.
Then 
\[
\length (R/(\frq{I}+\mf m^{[qq']})) \leq \length (R/(\frq{I}+(x_1, \ldots, x_h)^{[qq']})) = 
\length (R/(\frq{I}+(x_1^{qq'}, \ldots, x_h^{qq'}))).
\]
Filtering by the powers of $x_i^q$, we get that
\[
\length (R/(\frq{I}+(x_1^{qq'}, \ldots, x_h^{qq'}))) \leq (q')^h \length(R/\frq{(I, x_1, \ldots, x_h)}).
\]
Since $(I, x_1, \ldots, x_h)$ is an $\mf m$-primary ideal, it contains a system of parameters, say, 
$y_1, \ldots, y_d$.
Then, filtering by the powers of $y_i$, we obtain that
\[
\length(R/\frq{(I, x_1, \ldots, x_h)}) \leq \length (R/\frq{(y_1, \ldots, y_d)}) \leq q^d \length (R/(y_1, \ldots, y_d)).
\]
Last, let $C = \length (R/(y_1, \ldots, y_d))$ and observe that $d = \dim R$ and $h = \dim R/I$.
\end{proof}

\begin{corollary}\label{module bound}
Let $(R, \mf m)$ be a local ring, let $J$ be an $\mf m$-primary ideal, and $I$ be an arbitrary ideal.
If $M$ is a finitely generated $R$-module, 
then there exists a constant $D$ (independent of $q'$) such that for all $q, q'$
\[
\length (M/(\frq{I} + J^{[qq']})M) \leq D q'^{\dim R/I}q^{\dim M}.
\]
\end{corollary}
\begin{proof}
Since $J$ is $\mf m$-primary, $\mf m^{[q_0]} \subseteq J$ for some $q_0$, thus  
if the result holds for $J = \mf m$, then
\[
\length (M/(\frq{I} + J^{[qq']})M) \leq \length (M/(\frq{I} + \mf m^{[qq'q_0]})M)
\leq (Dq_0^{\dim M}) q'^{\dim R/I}q^{\dim M}.
\] 
Therefore we assume that $J = \mf m$.

Let $K$ be the annihilator of $M$ and let $n$ be the minimal number of generators of $M$.
Then there exists a surjection $(R/K)^n \to M \to 0$, so, after tensoring with $R/(\frq{I}+\mf m^{[qq']})$,
we obtain from Lemma~\ref{boundlemma} the estimate
\begin{align*}
\length (M/(\frq{I}+\mf m^{[qq']})M) &\leq n\length (R/(K + \frq{I}+\mf m^{[qq']}))
\leq nC q'^{\dim R/(I+K)}q^{\dim (R/K)} \\
&\leq nC q'^{\dim R/I}q^{\dim M}.
\end{align*}
\end{proof}

For the next lemma, recall that $\Minh(R)$ denotes the set of prime ideals $P$ such that 
$\dim R/P = \dim R$.

\begin{lemma}\label{min primes}
Let $(R, \mf m)$ be a reduced local ring, $J$ be an $\mf m$-primary ideal, and $I$ be an arbitrary ideal.
Let $M,N$ be finitely generated $R$-modules such that $M_P \cong  N_P$ for any $P \in \Minh (R)$.
Then there exists a constant $C$ independent of $q'$ and such that for all $q,q'$
\[
|\length (M/(\frq{I} + J^{[qq']})M) - \length (N/(\frq{I}+ J^{[qq']})N)| < Cq'^{\dim R/I} q^{\dim R - 1}.
\] 
\end{lemma}
\begin{proof}
Let $S = R \setminus \cup_{\Minh(R)} P$. As explained in \cite[Lemma~3.3]{Tucker}, 
there exist homomorphisms $M \to N$ and $N \to M$ that become isomorphisms
after localization by $S$. 
Thus we have exact sequences 
\[
M \to N \to K_1 \to 0,
\]
\[
N \to M \to K_2 \to 0
\]
where $\dim K_1, \dim K_2 < \dim R$, since $S^{-1}K_1 = S^{-1}K_2 = 0$.

Tensoring the first exact sequence with $R/(\frq{I} + J^{[qq']})$ and taking length, we obtain that 
\[
\length (N/(\frq{I} + J^{[qq']})N)  \leq \length (M/(\frq{I} + J^{[qq']})M) + \length (K_1/(\frq{I} + J^{[qq']})K_1),
\]
while the second sequence yields 
\[
\length (M/(\frq{I} + J^{[qq']})M)  \leq \length (N/(\frq{I} + J^{[qq']})N) + \length (K_2/(\frq{I} + J^{[qq']})K_2).
\]
By Corollary~\ref{module bound}, there are constants $C_1$ and $C_2$ such that
$\length (K_1/(\frq{I} + J^{[qq']})K_1) \leq C_1q'^{\dim R/I} q^{\dim R - 1}$
and $\length (K_2/(\frq{I} + J^{[qq']})K_2) \leq C_2q'^{\dim R/I} q^{\dim R - 1}$.
Combining the estimates together, we derive that
\[|\length (N/(\frq{I} + J^{[qq']})N) -\length (M/(\frq{I} + J^{[qq']})M)| \leq 
\max (C_1, C_2) q'^{\dim R/I} q^{\dim R - 1}.
\]
\end{proof}

For the next result, we need a bit of notation.
For an $R$-module $M$ and an integer $e$, we use $F_*^e M$ to denote an $R$-module obtained from $M$ 
via the extension of scalars through the $e$th iterate of the Frobenius endomorphism $F^e \colon R \to R$.
Thus $F_*^e M$ is isomorphic to $M$ as an abelian group, but elements of $R$ act as $p^e$-powers.
So, for any ideal $I$, $IF_*^e M \cong F_*^e \frq{I}M$. 
If $R$ is reduced, $F_* R$ can be identified with the ring of $p$-roots $R^{1/p}$.

\begin{definition}
Let $R$ be a ring of characteristic $p > 0$. 
For a prime ideal $\mf p$ of $R$, we denote 
$\alpha(\mf p) = \log_p [k(\mf p):k(\mf p)^p]$,
where $k(\mf p) = R_\mf p/\mf pR_\mf p$ is the residue field of $\mf p$.
\end{definition}

This number controls the change of length under Frobenius, 
$\length_R (F_*^e M) = p^{e\alpha(R)} \length_R (M)$ for a finite length $R$-module $M$.
Kunz (\cite[Proposition~2.3]{Kunz2}) observed how these numbers change under localization.

\begin{proposition}\label{kunzdegree}
Let $R$ be F-finite and let $\mf p \subseteq \mf q$ be prime ideals.
Then $\alpha(\mf p) = \alpha(\mf q) + \dim R_\mf q/\mf pR_\mf q$.
\end{proposition}

\begin{corollary}\label{rank of frobenius}
Let $(R, \mf m)$ be a reduced F-finite local ring of dimension $d$ and $M$ be a finitely generated $R$-module. 
Then for any $\mf p \in \Minh (R)$ of $R$ the modules 
$M_\mf p^{\oplus p^{\alpha(\mf m) + d}}$ and $(F_* M)_\mf p$ are isomorphic. 

In particular, if $R$ is equidimensional and $M$ has rank $r$, then the rank of $F_* M$ is $rp^{\alpha(\mf m) + d}$.
\end{corollary}
\begin{proof}
By Proposition~\ref{kunzdegree}, $\alpha(\mf p) = \alpha (\mf m) + d$ for any $\mf p \in \Minh(R)$.
Hence
\[
\length_{R_\mf p} (F_* (M_\mf p)) = p^{\alpha(\mf p)} \length_{R_\mf p}(M_\mf p).
\]
Since $R$ is reduced, $R_\mf p$ is a field, hence the vector spaces $F_* (M_\mf p)$ and 
$\oplus^{p^{\alpha(\mf p)}} M_\mf p$ are isomorphic. Last, recall that $F_* (M_\mf p) \cong (F_* M)_\mf p$ 
since $F_*$ commutes with localization.
\end{proof}

\begin{theorem}\label{F-finite convergence}
Let $(R, \mf m)$ be a reduced F-finite local ring of dimension $d$, $J$ be an $\mf m$-primary ideal, and $I$ be an arbitrary ideal.
Then for any finitely generated $R$-module $M$ there exists a constant $C$ such that
\[|\length (M/(I^{[q]}+J^{[qq']})M) - q^d\ehk(I + J^{[q']}, M)| < Cq^{d - 1}q'^{\dim R/I}\]
for all $q, q'$.
In particular, the bisequence
\[
\frac{\length (M/(\frq{I} + J^{[qq']})M)}{q^{\dim R}q'^{\dim R/I}}
\]
converges uniformly with respect to $q$.
\end{theorem}
\begin{proof}
By Corollary~\ref{rank of frobenius}, 
$(F_* M)_P$ and $M_P^{\oplus p^{\alpha(R) + d}}$ are isomorphic for any minimal prime $P \in \Minh (R)$.
Thus, we can apply Lemma~\ref{min primes} to $M^{\oplus p^{\alpha(R) + d}}$ and $F_* M$ and get that
\[
|p^{\alpha(R) + d}\length (M/(\frq{I} + J^{[qq']})M) - \length (F_* M/(\frq{I} + J^{[qq']})F_* M)| 
< Cq'^{\dim R/I} q^{d - 1}
\] 
for any $q, q'$ and a constant $C$ depending only on $M$ and $I$.
We have
\[
\length (F_* M/(\frq{I} + J^{[qq']})F_* M) = p^{\alpha(R)}\length (M/(I^{[qp]} + J^{[qpq']})M).
\]
Therefore, using that $p^{-\alpha(R)}C \leq C$, 
\begin{equation}\label{1 formula}
|p^d\length (M/(\frq{I} + J^{[qq']})M) - \length (M/(I^{[qp]} + J^{[qpq']})M)| <  Cq'^{\dim R/I} q^{d - 1}.
\end{equation}

Now, we prove by induction on $q_1$ that for all $q,q'$
\begin{equation}\label{ind formula}
|(q_1)^{d}\length (M/(I^{[q]} + J^{[qq']})M) - \length (M/(I^{[qq_1]} + J^{[qq_1q']})M)|<
Cq'^{\dim R/I}(qq_1/p)^{d -1}\frac{q_1 - 1}{p - 1}. 
\end{equation}
The induction base of $q_1 = p$ is (\ref{1 formula}). 
Now, assume that the claim holds for $q_1$ and we want to prove it for $q_1p$.

First, (\ref{1 formula}) applied to $qq_1$ gives 
\begin{equation}\label{2 formula}
|p^{d}\length (M/(I^{[qq_1]} + J^{[qq_1q']})M) - \length (M/(I^{[qq_1p]} + J^{[qq_1pq']})M)|
< Cq'^{\dim R/I}(qq_1)^{d - 1},
\end{equation}
and, multiplying the induction hypothesis by $p^{d}$, we get
\begin{equation}\label{3 formula}
|(q_1p)^{d}\length (M/(I^{[q]} + J^{[qq']})M) - 
p^{d}\length (M/(I^{[qq_1]} + J^{[qq_1q']})M)|
< Cq'^{\dim R/I}(qq_1)^{d -1}\frac{q_1p - p}{p - 1}.
\end{equation}
Using that $p^{d}\length (M/(I^{[qq_1]} + J^{[qq_1q']})M)$
appears in both (\ref{2 formula}) and (\ref{3 formula}), one obtains 
\begin{align*}
&|(q_1p)^{d}\length (M/(I^{[q]} + J^{[qq']})M) - \length (M/(I^{[qq_1p]} + J^{[qq_1pq']})M)| < \\ &<
Cq'^{\dim R/I}(qq_1)^{d -1} \left(\frac{q_1p - p}{p - 1} + 1\right)
 = Cq'^{\dim R/I}(qq_1)^{d -1} \left(\frac{q_1p - 1}{p - 1}\right),
\end{align*}
and the induction step follows.

Now, dividing (\ref{ind formula}) by $q_1^{d}$, we obtain
\[
|\length (M/(I^{[q]} + J^{[qq']})M) - \frac{1}{q_1^{d}}\length (M/(I^{[qq_1]} + J^{[qq_1q']})M)|<
Cq^{d - 1}\cdot\frac{q_1 - 1}{p -1}\cdot\frac{1}{q_1p^{d - 1}}
\leq Cq^{d - 1}.
\]

Thus, if we let $q_1 \to \infty$ and note that $\ehk (I^{[q]} + J^{[qq']}, M) = q^d \ehk(I + J^{[q']})$, we get that 
\[|\length (M/(I^{[q]} + J^{[qq']})M) - q^d\ehk(I + J^{[q']}, M)| < Cq^{d - 1},\]
and the claim follows.
\end{proof}

\begin{corollary}\label{uniform convergence}
Let $(R, \mf m)$ be a local ring of dimension $d$, $J$ an $\mf m$-primary ideal, and $I$ be an arbitrary ideal.
Then there is a $q_0$ such that for any finitely generated $R$-module $M$ there exists a constant $C$ such that
\[|\length (M/(I^{[q]} + J^{[qq']})M) - q^d\ehk(I + J^{[q']}, M)| < Cq^{d - 1}q'^{\dim R/I}\]
for all $q'$ and all $q \geq q_0$.
In particular, the bisequence
\[
\frac{\length (M/(\frq{I} + J^{[qq']})M)}{q^{\dim R}q'^{\dim R/I}}
\]
converges uniformly with respect to $q$.
\end{corollary}
\begin{proof}
First, we reduce to the case where $R$ is F-finite. Using the standard construction (\cite[Corollary~3.5]{Tucker}),
we can find a faithfully flat F-finite extension $S$ of $R$ such that $\mf mS$
is the maximal ideal of $S$. Hence, for any Artinian $R$-module $A$, $\length_S (A \otimes_R S) = \length_R (A)$.

Now, there is $q_0 = p^{e_0}$ such that $(\sqrt{0S})^{q_0} = 0$. Naturally, $N = F_*^{e_0} (S \otimes_R M)$
is a $\red{S}$-module, where $\red{S} = S/\sqrt{0S}$. 
Since $\red{S}$ is reduced and F-finite, we can apply Theorem~\ref{F-finite convergence}
and find a constant $C$ such that
\begin{equation}\label{ffinite equation}
|\length_S (\red{S}/(I^{[q]}+J^{[qq']})\red{S} \otimes_{\red{S}} N) - q^d\ehk ((I + J^{[q']})\red{S}, N)| 
< Cq^{d - 1}
\end{equation}
for all $q, q'$.

Now, observe that 
\[
\begin{split}
\frac{\red{S}}{(I^{[q]} + J^{[qq']})\red{S}} \otimes_{\red{S}} N \cong 
\frac{S}{(I^{[q]}+J^{[qq']})S} \otimes_{S} N &\cong
F_*^{e_0} \left (\frac{S \otimes_R M}{(I^{[qq_0]} + J^{[qq_0q']})(S \otimes_R M)} \right ) \\ &\cong 
F_*^{e_0} \left( \frac{M}{(I^{[qq_0]} + J^{[qq'q_0]})M} \otimes_R S\right).
\end{split}
\]
So, 
\[
\length_S \left( \frac{N}{(I^{[q]}+ J ^{[qq']})N}\right) = 
\length_S \left ( \frac{M}{(I^{[qq_0]} + J^{[qq'q_0]})M} \otimes_R S \right) = \length_R (M/(I^{[qq_0]} + J^{[qq'q_0]})M).\]
Therefore, by definition,
\[\ehk ((I + J^{[q']})\red{S}, N) = \ehk ((I + J^{[q']})^{[q_0]}, M) = q_0^d \ehk (I + J^{[q']}, M)\] 
and we can rewrite (\ref{ffinite equation}) as
\[
|\length_R (M/(I^{[qq_0]} + J^{[qq_0q']})M) - (qq_0)^d\ehk (I + J^{[q']}, M )| 
< Cq^{d - 1} \leq C(qq_0)^{d-1}.
\]
By setting $q = qq_0$, we get that for all $q \geq q_0$ and all $q'$
\[
|\length_R (M/(I^{[q]} + J^{[qq']})M) - q^d\ehk (I + J^{[q']},M )| 
< Cq^{d - 1}.
\]
\end{proof}

Now, we can establish the main result of this section, which will be the basic tool of our theory of equimultiplicity.

\begin{corollary}\label{ehk limits}
Let $(R, \mf m)$ be a local ring, and $J$ be an $\mf m$-primary ideal. 
If $I$ is an ideal such that $\dim R/I + \hght I = \dim R$, then
\[
\lim_{q' \to \infty} \frac{\ehk(I + J^{[q']}, M)}{q'^{\dim R/I}} = 
\sum_{P \in \Minh (I)} \ehk(JR/P, R/P) \ehk(IR_P, M_P).
\]
\end{corollary}
\begin{proof}
We have proved that the double sequence
\[
\frac{\length (M/(\frq{I} + J^{[qq']})M)}{q^{\dim R}q'^{\dim R/I}}
\]
converges uniformly with respect to $q$.
Moreover, the limit with respect to $q'$ exists for any $q$ since
\[
\lim_{q'\to \infty} \frac{\length (M/(\frq{I} + J^{[qq']})M)}{q^{\dim M}q'^{\dim R/I}}
= \frac {\ehk (\frq{J}R/\frq{I}, M/\frq{I}M)}{q^{\dim R}},
\]
where $\ehk (\frq{J}R/\frq{I}, M/\frq{I}M)$ is taken over the ring $R/\frq{I}$.
Thus, the iterated limits of the double sequence are equal:
\[
\begin{split}
\lim_{q' \to \infty} \frac{\ehk (I+J^{[q']}, M)}{q'^{\dim R/I}} &= 
\lim_{q' \to \infty} \lim_{q \to \infty} \frac{\length (M/(\frq{I} + J^{[qq']})M)}{q^{\dim R}q'^{\dim R/I}} 
\\ = \lim_{q \to \infty} &\lim_{q' \to \infty} \frac{\length (M/(\frq{I} + J^{[qq']})M)}{q^{\dim R}q'^{\dim R/I}}
= \lim_{q \to \infty} \frac {\ehk(\frq{J}R/\frq{I}, M/\frq{I}M)}{q^{\dim R}}.
\end{split}
\] 

Note that $\sqrt{I} = \sqrt{\frq{I}}$, so $\dim R/I = \dim R/\frq{I}$ and $\Minh(I) = \Minh(\frq{I})$. 
Moreover, by the additivity formula, 
\[
\ehk (JR/\frq{I}, M/\frq{I}M) = 
\sum_{P \in \Minh (I)} \ehk (JR/P, R/P) \length_{R_P} (M_P/\frq{I}M_P).
\]
Hence, we may use that 
$\ehk (\frq{J}R/\frq{I}, M/\frq{I}M) = q^{\dim R/I}\ehk (JR/\frq{I}, M/\frq{I}M)$
and obtain that
\begin{align*}
\lim_{q \to \infty} \frac {\ehk(\frq{J}R/\frq{I}, M/\frq{I}M)}{q^{\dim R}} &= 
\lim_{q \to \infty} \frac {\ehk(JR/\frq{I}, M/\frq{I}M)}{q^{\hght I}} \\&= 
\sum_{P \in \Minh (I)} \ehk(JR/P, R/P) \lim_{q \to \infty } \frac{\length_{R_P} (M_P/\frq{I}M_P)}{q^{\hght I}}.
\end{align*}
The claim follows, since $\hght I = \hght P$.

\end{proof}

\begin{corollary}\label{asymptotic descent}
Let $(R, \mf m)$ be a local ring, $J$ an $\mf m$-primary ideal, and $\mf p$ be a prime ideal
such that $\dim R/\mf p + \hght \mf p = \dim R$.
Then
\[
\lim_{q \to \infty} \frac{\ehk(\mf p + J^{[q]})}{q^{\dim R/\mf p}} = \ehk(JR/\mf p, R/\mf p) \ehk(R_\mf p).
\]
\end{corollary}

When $R/\mf p$ is regular, this corollary will help us to connect $\ehk(\mf m)$ to 
$\ehk(\mf p)$.

\section{Equimultiplicity theory}\label{HK equi general}

In this section, we study {\it equimultiple ideals} for Hilbert--Kunz multiplicity.
We will show that these should be ideals $I$ 
such that for any (or, as we will show, some) ideal $J = (x_1, \ldots, x_m)$ generated by a
system of parameters modulo $I$
\begin{equation}\label{Equi equation}
\ehk (I + J) = \sum_{P \in \Minh(I)} \ehk (J, R/P) \ehk (I, R_P).
\end{equation}
This could be seen as a direct analogue of condition (b) of Theorem~\ref{HS equimultiple}.

Before proceeding further let us record the following lemma that shows that
Equation~\ref{Equi equation} is an extremal condition.

\begin{lemma}\label{HK inequality}
Let $(R, \mf m)$ to be a local ring of characteristic $p > 0$ and 
let $I$ be an ideal such that $\hght I + \dim R/I = \dim R$.
Then for any parameter ideal $J$ modulo $I$
\[
\ehk(I + J) \geq \sum_{P \in \Minh (I)} \ehk(J, R/P)\ehk(IR_P).
\]
\end{lemma}
\begin{proof}
First, by Proposition~\ref{hs param} and Corollary~\ref{multiplicity of param power},
\[
\length (R/\frq{(I + J)}) \geq \eh (\frq{J}, R/\frq{I}) = 
q^{\dim R/I} \eh(J, R/\frq{I}).
\]
So, by the additivity formula,
\[
\ehk(I + J) \geq \lim_{q \to \infty} \sum_{P \in \Minh(I)} \frac 1{q^{\hght I}} \eh(J, R/P)\length (R_P/\frq{I}R_P) 
\]
and the claim follows.
\end{proof}

\subsection{Preliminaries}
The main ingredient of our proof is the multi-sequence
\[
\lim \limits_{\min(n_i) \to \infty} 
\frac{1}{n_1\cdots n_m}\ehk(I, x_1^{n_1}, \ldots, x_m^{n_m})
\]
that will allow us to connect the two sides of Equation~\ref{Equi equation}.
First, let us record two useful result in the case when $\dim R/I = 1$.

\begin{lemma}\label{nmonotone}
Let $(R, \mf m)$ be a local ring, $I$ be an ideal, and $x$ be a parameter modulo $I$.
Suppose $\dim R/I = \dim R - 1$.
Then
\[\frac{1}{n}\length (R/(I, x^n)) \geq \frac{1}{n+1}\length (R/(I, x^{n + 1})).\]
\end{lemma}
\begin{proof}
Observe that 
\[
\frac{I + (x^k)}{I + (x^{k + 1})} \cong \frac{(x^k)}{(x^k)\cap I + (x^{k + 1})} \cong \frac{R}{I:x^k + (x)}.
\]
Thus we get the formula
\[\length \left( \frac{R}{(I, x^{k + 1})} \right) =   
\length\left(\frac{R}{(I, x^k)}\right) + \length\left(\frac{R}{I:x^k + (x)}\right).
\]
First of all, setting $n = k$ in the formula, we see that it is enough to show that 
$\length (R/(I, x^n)) \geq n\length \left(R/(I:x^n + (x)) \right)$.
Second, using the formula above for consecutive values of $k$, we obtain that 
\[\length \left (R/(I, x^n)\right) = \length \left (R/(I, x)\right) + 
\sum_{k = 1}^{n-1} \length \left (R/(I:x^k + (x))\right) \geq n \length \left (R/(I:x^{n} + (x))\right ), \] 
where the last inequality holds since $I:x^k \subseteq I:x^n$ for all $0 \leq k \leq n$.
\end{proof}

\begin{corollary}\label{hkmonotone}
In the setting of the lemma, we have
\[\frac{1}{n} \ehk (I, x^n) \geq \frac{1}{n+1} \ehk(I, x^{n+1}).\]
\end{corollary}
\begin{proof}
Apply the lemma to $\frq I$ and $x^q$ and take the limit as $q \to \infty$.
\end{proof}

Now, we can show that our multi-sequence converges to the right-hand side of Equation~\ref{Equi equation}.

\begin{proposition}\label{general limit}
Let $(R, \mf m)$ to be a local ring of characteristic $p > 0$ and 
let $I$ be an ideal in $R$ such that $\hght I = \dim R - \dim R/I$.
If $x_1, \ldots, x_m$ are a system of parameters modulo $I$ then
\[
\lim \limits_{\min(n_i) \to \infty} 
\frac{1}{n_1\cdots n_m}\ehk(I, x_1^{n_1}, \ldots, x_m^{n_m}) 
= \sum_{P \in \Minh(I)} \ehk ((x_1, \ldots, x_m), R/P)\ehk(IR_P).
\]
\end{proposition}
\begin{proof}
Let $(n_1, \ldots, n_m) \in \mathbb N^m$ be an arbitrary vector and let $n = \min (n_1, \ldots, n_m)$
and $N = \max (n_1, \ldots, n_m)$. Corollary~\ref{hkmonotone} shows that 
\[
\frac{1}{N^m}\ehk(I, x_1^{N}, \ldots, x_m^{N}) \leq
\frac{1}{n_1\cdots n_m}\ehk(I, x_1^{n_1}, \ldots, x_m^{n_m}) \leq 
\frac{1}{n^m}\ehk(I, x_1^{n}, \ldots, x_m^{n}).
\]
Therefore,  
\[
\lim \limits_{\min(n_i) \to \infty} 
\frac{1}{n_1\cdots n_m}\ehk(I, x_1^{n_1}, \ldots, x_m^{n_m}) 
= \lim \limits_{n \to \infty} 
\frac{1}{n^m}\ehk(I, x_1^{n}, \ldots, x_m^{n}).
\]
Moreover, by Corollary~\ref{hkmonotone}, the sequence
$\frac{1}{n^m}\ehk(I, x_1^{n}, \ldots, x_m^{n})$ is monotonically decreasing, 
so its limit exists and can be computed by looking at a subsequence. 
But, by Corollary~\ref{ehk limits},
\[
\lim_{q' \to \infty}\frac{1}{{q'}^m}\ehk(I, x_1^{q'}, \ldots, x_m^{q'}) = 
\sum_{P \in \Minh(I)} \ehk ((x_1, \ldots, x_m), R/P)\ehk(IR_P).\] 
\end{proof}

\begin{corollary}\label{monotone equality}
In the assumptions of Proposition~\ref{general limit}, if 
\[
\ehk (I, x_1, \ldots, x_m) = \sum_{P \in \Minh(I)} \ehk ((x_1, \ldots, x_m), R/P)\ehk(IR_P),
\]
then for any vector $(n_1, \ldots, n_m) \in \mathbb N^m$
\[
\ehk(I, x_1^{n_1}, \ldots, x_m^{n_m}) = \sum_{P \in \Minh(I)} \ehk ((x_1^{n_1}, \ldots, x_m^{n_m}), R/P)\ehk(IR_P).
\]
\end{corollary}
\begin{proof}
By Corollary~\ref{hkmonotone},
\[
\ehk (I, x_1, \ldots, x_m) \geq 
\frac {\ehk(I, x_1^{n_1}, \ldots, x_m^{n_m})}{n_1\cdots n_m}. 
\]
Moreover, by Lemma~\ref{HK inequality} and Corollary~\ref{multiplicity of param power},
\begin{align*}
\frac {\ehk(I, x_1^{n_1}, \ldots, x_m^{n_m})}{n_1\cdots n_m} \geq
\sum_{P \in \Minh(I)} \frac{\ehk ((x_1^{n_1}, \ldots, x_m^{n_m}), R/P)}{n_1\cdots n_m}\ehk(IR_P) \\= 
\sum_{P \in \Minh(I)} \ehk ((x_1, \ldots, x_m), R/P)\ehk(IR_P)
= \ehk (I, x_1, \ldots, x_m).
\end{align*}
\end{proof}

For the next corollary recall the following observation (\cite[Lemma~4.2]{WatanabeYoshida}).

\begin{lemma}\label{HK inclusions}
Let $(R, \mf m)$ be a local ring of characteristic $p > 0$ and $I$ an $\mf m$-primary ideal. 
Then $\ehk (I) \leq \length (R/I) \ehk(\mf m)$.
\end{lemma}

\begin{corollary}\label{gcriterion}
Let $(R, \mf m)$ to be a local ring of characteristic $p > 0$ and 
$\mf p$ be a prime ideal in $R$ such that $R/\mf p$ is a regular local ring and
$\hght \mf p = \dim R - \dim R/\mf p$. 
Then the following are equivalent:
\begin{enumerate}
\item[(i)]
$\ehk (\mf m) = \ehk(\mf p)$, 
\item[(ii)] $\ehk (\mf p, x_1^{n_1}, \ldots, x_m^{n_m}) = n_1\cdots n_m\ehk (\mf m)$
for all vectors $(n_1, \ldots, n_m)$ and any system of minimal generators $x_1, \ldots, x_n$ of $\mf m$ modulo $\mf p$, 
\item[(iii)] $\ehk (\mf p + I) = \length (R/(\mf p + I))\ehk (\mf m)$
for any system of parameters $I$ modulo $\mf p$.
\end{enumerate}
\end{corollary}
\begin{proof}
Clearly, (iii) $\Rightarrow$ (ii) $\Rightarrow$ (i).
For (i) $\Rightarrow$ (iii), we observe by Lemma~\ref{HK inequality} that
$\ehk (\mf p + I) \geq \length (R/(\mf p + I))\ehk (\mf p)$.
On the other hand, the preceding remark gives that
\[
\ehk(\mf p + I) \leq \length (R/(\mf p + I))\ehk(\mf m) = \length (R/(\mf p + I))\ehk(\mf p).
\]
\end{proof}

After building some machinery, we are going to show (Corollary~\ref{general HK equimultiplicity}) that
the equality $\ehk (\mf p + I) = \length (R/(\mf p + I))\ehk (\mf p)$ for some system of parameters $I$ 
forces that $\ehk(\mf m) = \ehk(\mf p)$.

\subsection{Main results}

The next fundamental theorem can be seen as an analogue of implication $(b) \Rightarrow (d)$
of Theorem~\ref{HS equimultiple}.	
	
\begin{theorem}\label{coloninc2}
Let $(R, \mf m)$ be a formally unmixed local ring of characteristic $p > 0$ with a test element $c$.
Let $I$ be an ideal and suppose for some parameter ideal $J = (x_1, \ldots, x_m)$ modulo $I$,
\[
\ehk (I + J) = \sum_{P \in \Minh(I)} \ehk (J, R/P) \ehk (I, R_P).
\]
Then 
$\frq{(I, x_1, \ldots, x_{i-1})} : x_{i}^{\infty} \subseteq \left( \frq{(I, x_1, \ldots, x_{i-1})} \right)^*$
for all $q$ and $1 \leq i \leq m$.

In particular, $x_{i}$ is not a zero divisor modulo $\left (\frq{(I, x_1, \ldots, x_{i-1})} \right)^*$ 
for all $i$ and $q$.
\end{theorem}
\begin{proof}
First, observe that the second claim follow from the first via Lemma~\ref{tc nzd}.
	
Let $d$ be the dimension of $R$. For a fixed $k$, let $L = (x_1, x_2, \ldots, x_{i-1}, x_{i+1}^k, \ldots, x_m^k)$.	
For any $n, q, q'$ we have inclusions
\begin{equation}\label{colon equation}
\left(I, L, x_{i}^k \right)^{[qq']} \subseteq 
\left (\frq{(I, x_1, \ldots, x_{i-1})}:x_{i}^{nq} \right)^{[q']} + \left(L, x_{i}^k\right)^{[qq']}
\subseteq  \left(I, L\right)^{[qq']}:x_{i}^{nqq'} + \left(x_{i}^{kqq'}\right). 
\end{equation}
Hence, after dividing by $q'^{d}$ and taking the limit, we obtain that
\begin{equation}\label{ehk equation}
\begin{split}
\ehk\left(\frq{(I, L, x_{i}^k)}\right)
&\geq \ehk\left( \frq{(I, x_1, \ldots, x_{i-1})}:x_{i}^{nq} + \frq{(L, x_{i}^k)} \right) \\ &\geq
\lim_{q'\to \infty} \frac{1}{(q')^{d}} \length \left( \frac{R}{(I, L)^{[qq']}:x_{i}^{nqq'} + (x_{i}^{kqq'})}\right). 
\end{split}
\end{equation}

By Corollary~\ref{monotone equality} and Corollary~\ref{multiplicity of param power}, for all $n$
\[
\ehk (I, L, x_i^n) = n \sum_{P \in \Minh(I)} \eh ((x_i, L), R/P) \ehk (I, R_P).
\]
In particular, we obtain that
\[
\ehk\left(\frq{(I, L, x_{i}^k)}\right) = q^{d}\ehk\left(I, L, x_{i}^k\right) 
= kq^{d}\smashoperator[r]{\sum_{P \in \Minh(I)}} \eh ((x_i, L), R/P) \ehk (I, R_P).
\]
Moreover, from the isomorphism 
$R/\left (\frq{(I, L)}:x_{i}^{nq}, x_{i}^{kq} \right) \cong \frq{(I, L, x_{i}^{n})}/\frq{(I, L, x_{i}^{n + k})}$, 
we get an exact sequence
\[
0 \to 
R/(\frq{(I, L)}:x_{i}^{nq}, x_{i}^{kq}) 
\to 
R/\frq{(I, L, x_{i}^{n + k})}  \to R/\frq{(I, L, x_{i}^{n})} 
\to 0. 
\]
Together with the previous computation, the sequence gives that for all $n$ and $k$
\[
\lim_{q \to \infty} \frac {1}{q^{d}} \length \left( \frac{R}{\frq{(I, L)}:x_{i}^{nq} + (x_{i}^{kq})} \right) 
= k \sum_{P \in \Minh(I)} \eh ((x_i, L), R/P) \ehk (I, R_P),
\]
so, we may compute
\[
\lim_{(qq')\to \infty} \frac{q^{d}}{(qq')^{d}}
\length \left( \frac{R}{(I, L)^{[qq']}:x_{i}^{nqq'} + (x_{i}^{kqq'})} \right) = 
kq^{d}\smashoperator[lr]{\sum_{P \in \Minh(I)}}\eh ((x_i, L), R/P) \ehk (I, R_P).
\]

Thus, by (\ref{colon equation}) and (\ref{ehk equation})
\[
\ehk\left( \frq{(I, x_1, \ldots, x_{i-1})}:x_{i}^{nq}, \frq{(L, x_{i}^k)} \right) = 
\ehk\left(\frq{(I, L, x_{i}^k)}\right).
\]
Therefore, by Theorem~\ref{HH TC},
$\frq{(I, x_1, \ldots, x_{i-1})}:x_{i}^{nq} \subseteq \left (\frq{(I, L, x_{i}^k)} \right)^*.$
Now, since $n$ is arbitrary, we have
\[
\frq{(I, x_1, \ldots, x_{i-1})}:x_{i}^{\infty} \subseteq 
\bigcap_k \left (\frq{(I, x_{1}, \ldots, x_{i-1}, x_i^k, \ldots, x_m^k)} \right)^*
\]
and the assertion follows from Lemmas~\ref{intersection} and \ref{tc nzd}. 
\end{proof}

Now, we will establish the converse to Theorem~\ref{coloninc2}.

\begin{theorem}\label{capturing equivalence}
Let $(R, \mf m)$ be a formally unmixed local ring of characteristic $p > 0$ with a locally stable test element $c$.
Let $I$ be an ideal and $J = (x_1, \ldots, x_m)$ be an ideal generated by a system of parameters modulo $I$.
Then
\[
\ehk (I + J) = \sum_{P \in \Minh(I)} \ehk (J, R/P) \ehk (I, R_P)
\]
if and only if $x_{i}$ is not a zero divisor modulo $(\frq{(I, x_1, \ldots, x_{i-1})})^*$  for all $i$ and $q$.
\end{theorem}
\begin{proof}
One direction follows from Theorem~\ref{coloninc2}.
For the converse, we use induction on $m$. If $m = 1$, then by Proposition~\ref{hs param} and 
by the additivity formula (Proposition~\ref{ass formula})
\[
\length \left(R/((\frq{I})^*, x_1^q) \right) = q\eh (x_1, R/(\frq{I})^*) = 
q\sum_{P \in \Minh(I)} \eh (x_1, R/P) \length (R_P/(\frq{I})^*R_P).
\]
So, taking the limit via Lemma~\ref{tcinvariance}, 
$\ehk (I, x_1) \leq \sum_{P \in \Minh(I)} \eh (x_1, R/P) \ehk (I, R_P)$
and the converse holds by Lemma~\ref{HK inequality}.
 
For $m > 1$, by the induction hypothesis, 
\[
\ehk (I + J) = \ehk ((I, x_1) + (x_2, \ldots, x_m)) 
= \smashoperator[lr]{\sum_{Q \in \Minh((I, x_1))}} \ehk ((x_2, \ldots, x_m), R/Q) \ehk ((I, x_1), R_Q).
\]
Since $c$ is locally stable, by Corollary~\ref{tc sequence limit} 
$((\frq{I})^*, x_1^q)R_Q$ still can be used to compute $\ehk ((I, x_1), R_Q)$. 
Thus, same way as in the first step,  we obtain 
\[
\ehk ((I, x_1), R_Q) = \sum_{P \in \Minh (IR_Q)} \eh (x_1, R_Q/PR_Q) \ehk (I, R_P).
\]
Combining these results, we get
\[
\ehk (I + J) 
= \sum_{Q \in \Minh((I, x_1))} \eh ((x_2, \ldots, x_m), R/Q) \sum_{P \in \Minh (IR_Q)} \eh (x_1, R_Q/PR_Q) \ehk (I, R_P).
\]
Observe that $\Minh (IR_Q) = \Min (IR_Q)$, since $x_1$ is a parameter modulo $I$ and $R_Q/IR_Q$ has dimension $1$.
Hence, any prime $P \in \Minh (IR_Q)$ is just a minimal prime of $I$ contained in $Q$.
Therefore, we can change the order of summations to get
\[
\ehk (I + J) = \sum_{P \in \Min(I)} \ehk (I, R_P) 
\sum_{P \subset Q \in \Minh ((I, x_1))} \eh (x_1, R_Q/PR_Q) \eh ((x_2, \ldots, x_m), R/Q),
\]
where the second sum is taken over all primes $Q \in \Minh (I + (x_1))$ that contain $P$.
For such $Q$ we must have $\dim R/P \geq \dim R/Q + 1 = \dim R/I$, because $x_1$ is a parameter modulo $I$.
So, in fact, the first sum could be taken over $\Minh (I)$. 
Furthermore, since $x_1$ is a parameter modulo $I$ and $P \in \Minh(I)$, $x_1$ is a parameter modulo $P$ and 
$\dim R/(I, x_1) = \dim R/(P, x_1)$.
Hence, the second sum is taken over $Q \in \Minh ((P, x_1))$, and we rewrite the formula as
\[
\ehk (I + J) = \sum_{P \in \Minh(I)} \ehk (I, R_P) 
\sum_{Q \in \Minh ((P, x_1))} \eh (x_1, R_Q/PR_Q) \eh ((x_2, \ldots, x_m), R/Q).
\]
Last, by the additivity formula for the multiplicity of a parameter ideal (Proposition~\ref{param ass formula}), for any $P$
\[
\ehk (J, R/P) = \sum_{Q \in \Minh((P, x_1))} \eh ((x_2, \ldots, x_m), R/Q) \eh (x_1, R_Q/PR_Q),
\]
and the claim follows.
\end{proof}

\begin{corollary}\label{part of J}
Let $(R, \mf m)$ be a formally unmixed local ring of positive characteristic with a locally stable test element $c$.
Let $I$ be an ideal and $J = (x_1, \ldots, x_m)$ be an ideal generated by a system of parameters modulo $I$.
If 
\[
\ehk (I + J) = \sum_{P \in \Minh(I)} \ehk (J, R/P) \ehk (I, R_P)
\]
then for any $0 \leq k \leq m$
\[
\ehk (I + J) = \sum_{Q \in \Minh ((I,x_1, \ldots, x_k))} \ehk ((x_{k+1}, \ldots, x_m), R/P) \ehk ((I, x_1, \ldots, x_k), R_Q).
\]
\end{corollary}
\begin{proof}
First, by Theorem~\ref{coloninc2},
$x_{i}$ is not a zero divisor modulo $(\frq{(I, x_1, \ldots, x_{i-1})})^*$  for all $i$ and $q$.
Now, since this holds for all $i \geq k$, Theorem~\ref{capturing equivalence} shows the assertion.
\end{proof}

\begin{corollary}\label{shuffle}
Let $(R, \mf m)$ be a local ring and let $x_1, \ldots, x_d$ and $y_1, \ldots, y_d$ be two systems of parameters.
Then there exists a linear combination $x' = x_d + a_1x_1 + \ldots + a_{d-1}x_{d-1}$ with coefficients in $R$ 
such that $x_1, \ldots, x_{d-1}, x'$ and $y_1, \ldots, y_{d-1}, x'$ are systems of parameters.
\end{corollary}
\begin{proof}
First, it is easy to see that $x_1, \ldots, x_{d-1}, x'$  
is still a system of parameters for any choice of the coefficients $a_i$.

Second, we let $P_1, \ldots, P_n$ be the minimal primes of $(y_1, \ldots, y_{d-1})$
and apply the avoidance lemma for cosets to $x= x_d$ and $I = (x_1, \ldots, x_{d-1})$.
\end{proof}

After all the hard work, we can establish that our definition of an equimultiple ideal is independent
on the choice of a parameter ideal.

\begin{proposition}\label{colon capture}
Let $(R, \mf m)$ be a formally unmixed local ring of characteristic $p > 0$
with a locally stable test element $c$ and let $I$ be an ideal.  
If for some parameter ideal $J = (x_1, \ldots, x_m)$ modulo $I$
\[
\ehk (I + J) = \sum_{P \in \Minh(I)} \ehk (J, R/P) \ehk (I, R_P),
\]
then same is true for all systems of parameters.
\end{proposition}
\begin{proof}
We use induction on $m$ and Theorem~\ref{capturing equivalence}.
 If $\dim R/I = 1$, then, by Theorem~\ref{coloninc2}, our assumption shows that 
$R/(\frq{I})^*$ is Cohen-Macaulay for any $q$, so any parameter is regular.

Let $(y_1, \ldots, y_m)$ be an arbitrary system of parameters modulo $I$.
Then using Corollary~\ref{shuffle}, 
we can find an element of the form $x' = x_{d} + a_1x_1 + \ldots + a_{m-1}x_{m-1}$
such that $y_1, \ldots, y_{m-1}, x'$ is still a system of parameters modulo $I$. 
Note that $(x_1, \ldots, x_{m-1}, x') = (x_1, \ldots, x_m) = J$, so the original formula still holds.
By Corollary~\ref{part of J}, we get that 
\[
\ehk (I + J) = \ehk (I, x', x_1, \ldots, x_{m-1}) = 
\smashoperator[lr]{\sum_{Q \in \Minh ((I, x'))}} \ehk ((x_1, \ldots, x_{m-1}), R/Q) \ehk ((I,x'), R_Q),
\]
so, by the induction hypothesis applied to $(I, x')$,
\[
\ehk (I, x', y_1, \ldots, y_{m-1}) = 
\smashoperator[r]{\sum_{Q \in \Minh ((I, x'))}} \ehk ((y_1, \ldots, y_{m-1}), R/Q) \ehk ((I, x'), R_Q).
\]
Using Theorem~\ref{coloninc2} on $(I, x')$, we get that 
$y_i$ is regular modulo $(\frq{(I, x', y_1, \ldots, y_{i-1})})^*$ for any $i$ and $q$. 
But since $x'$ is also regular modulo $(\frq{I})^*$ for all $q$, Theorem~\ref{capturing equivalence}
implies that 
\[
\ehk (I, x', y_1, \ldots, y_{m-1}) = \sum_{P \in \Minh(I)} \ehk ((x', y_1, \ldots, y_{m-1}), R/P) \ehk (I, R_P).
\]
After permuting the sequence and using Theorem~\ref{coloninc2}, we see that
$x'$ is not a zero divisor modulo $(\frq{(I, y_2, \ldots, y_n)})^*$ for all $q$.
Now, again, both $x'$ and $y_m$ are parameters modulo $(\frq{(I, y_1, \ldots, y_{m-1})})^*$, so $y_m$ is regular too.
\end{proof}

Motivated by Proposition~\ref{colon capture} and Theorem~\ref{capturing equivalence}, we 
introduce the following definition.

\begin{definition}\label{capturing def}
Let $(R, \mf m)$ be a local ring and let $I$ be an ideal.
We say that $I$ satisfies {\it colon capturing}, 
if for every system of parameters $x_1, \ldots, x_m$ in $R/I$, for every $0 \leq i < m$, and every $q$
\[
(\frq{(I, x_1, \ldots, x_i)})^* : x_{i + 1} \subseteq (\frq{(I, x_1, \ldots, x_i)})^*.
\]
\end{definition}

A well-known result of tight closure theory asserts that under mild conditions, $(0)$ satisfies colon capturing.
We note that the tight closure is taken in $R$, so this property is different from colon capturing in $R/I$. 

With this definition, we can summarize our findings in an analogue of the equivalence (b) and (d) 
of Theorem~\ref{HS equimultiple}.

\begin{theorem}\label{HK equimultiple}
Let $(R, \mf m)$ be a formally unmixed local ring of characteristic $p > 0$
with a locally stable test element $c$ and let $I$ be an ideal.  
Then the following are equivalent:
\begin{enumerate}
\item $I$ satisfies colon capturing,
\item for some (equivalently, every) ideal $J$ generated by a system of parameters modulo $I$,
\[
\ehk (I + J) = \sum_{P \in \Minh(I)} \ehk (J, R/P) \ehk (I, R_P).
\]
\end{enumerate}
\end{theorem}
\begin{proof}
This was proved in Theorem~\ref{capturing equivalence} and Proposition~\ref{colon capture}.
\end{proof}

In the special case of prime ideals with regular quotients we obtain the following characterization.

\begin{corollary}\label{general HK equimultiplicity}
Let $(R, \mf m)$ be an formally unmixed local ring of characteristic $p > 0$ with a locally stable test element $c$ and 
let $\mf p$ be a prime ideal such that $R/\mf p$ is a regular local ring.
Then the following are equivalent:
\begin{itemize}
\item[(a)] $\ehk(\mf m) = \ehk (\mf p)$,
\item[(b)] 
For any (equivalently, some) parameter ideal $J$ modulo $\mf p$ 
\[
\ehk (\mf p, J) = \length (R/(\mf p, J)) \ehk (\mf m),
\]
\item[(c)]
For any (equivalently, some) parameter ideal $J$ modulo $\mf p$ 
\[
\ehk (\mf p, J) = \length (R/(\mf p, J)) \ehk (\mf p),
\]
\item[(d)]
$\mf p$ satisfies colon capturing.
\end{itemize}
\end{corollary}
\begin{proof}
The first two conditions are equivalent by Corollary~\ref{gcriterion},
(a), (c), (d) are equivalent by the previous theorem.
\end{proof}

This theorem has a notable corollary. 
First, recall that a ring $R$ of characteristic $p > 0$ is called weakly F-regular if $I^* = I$ for every ideal $I$ in $R$. 
For example, any regular ring is weakly F-regular and direct summands of weakly F-regular rings are weakly F-regular.

\begin{corollary}\label{weakly f-regular}
Let $(R, \mf m)$ be a weakly F-regular excellent local domain and let $\mf p$ be a prime ideal such that $R/\mf p$ is regular.
Then $\ehk (\mf m) = \ehk(\mf p)$ if and only if the Hilbert--Kunz functions of $R$ and $R_\mf p$ coincide. 
\end{corollary}
\begin{proof}
Since all ideals in $R$ are tightly closed, from the preceding theorem 
we obtain that $R/\frq{\mf p}$ is Cohen-Macaulay for all $q$.
Hence the assertion follows from Proposition~\ref{HK functions}.
\end{proof}

\begin{remark}
In \cite[Lemma~4.2, Remark~4.7]{Ma}, Linquan Ma observed that perfect ideals are equimultiple and, thus, 
satisfy the colon capturing property. This observation is a crucial ingredient of his recent work on Lech's conjecture (\cite{Ma}). 
\end{remark}
	
\subsection{Further improvements}\label{Further}

We will develop some general reductions for the equimultiplicity condition
and use them to generalize the obtained results.

First, we show that equimultiplicity can be checked modulo minimal primes.

\begin{lemma}\label{prime mod p}
Let $(R, \mf m)$ be a local ring of characteristic $p > 0$ and $\mf p$ be a prime ideal such that 
$\hght \mf p + \dim R/\mf p = \dim R$.
Then $\ehk(\mf m) = \ehk(\mf p)$ if and only if 
$\Minh (R) = \Minh(R_\mf p)$ and $\ehk(\mf m R/P) = \ehk (\mf p R/P)$ for any $P \in \Minh (R)$.

In particular, if $R$ is catenary, 
then $\ehk(\mf m) = \ehk(\mf p)$ if and only if 
$P \subseteq \mf p$ and $\ehk(\mf m R/P) = \ehk (\mf p R/P)$ for all $P \in \Minh (R)$.
\end{lemma}
\begin{proof}
If $Q \in \Minh (R_\mf p)$, by definition, $\dim R_\mf p/QR_\mf p = \hght \mf p$,
so  $Q \in \Minh (R)$ by the assumption on $\mf p$. 
Moreover, if $R$ is catenary, it is easy to check that, in fact, 
$\Minh (R_\mf p) = \{P \in \Minh(R)\mid P \subseteq \mf p\}$.

By the additivity formula we have:
\[
\ehk(\mf m) = \sum_{P \in \Minh (R)} \ehk(\mf m, R/P) \length (R_P),
\]
and, by the localization property of Hilbert--Kunz multiplicity (\cite[Proposition~3.3]{Kunz2}),
\[
\ehk(\mf p) = \sum_{Q \in \Minh (R_\mf p)} \ehk(\mf p, R_\mf p/QR_\mf p) \length (R_Q)
\leq \sum_{Q \in \Minh (R_\mf p)} \ehk(\mf m, R/Q) \length (R_Q).
\]
Since the second sum is contained in the sum appearing in the expression for $\ehk(\mf m)$, the claim follows.
\end{proof}

The lemma can be easily generalized, but we decided to leave the special case for clarity.
A more general lemma can be found right after the following easy corollary.

\begin{corollary}\label{reduced equimultiplicity}
Let $(R, \mf m)$ be a local ring of characteristic $p > 0$ and $\mf p$ be a prime ideal such that 
$\hght \mf p + \dim R/\mf p = \dim R$.
Then $\ehk(\mf m) = \ehk(\mf p)$ if and only if $\ehk(\mf m\red{R}) = \ehk(\mf p\red{R})$.
\end{corollary}
\begin{proof}
Since $\Minh (R) = \Minh(\red{R})$, this immediately follows from the previous lemma applied to $R$ and $\red{R}$.
\end{proof}

\begin{lemma}\label{equi mod min}
Let $(R, \mf m)$ be a local ring of characteristic $p > 0$ and $I$ be an ideal such that 
$\hght I + \dim R/I = \dim R$. Let $J$ be a parameter ideal modulo $I$.
Then 
\[
\ehk (I + J) = \sum_{Q \in \Minh(I)} \ehk(J, R/Q) \ehk(IR_Q)
\]
if and only if the following two conditions hold:
\begin{itemize}
\item[(a)] $\Minh(I + P) \subseteq \Minh (I)$ for all $P \in \Minh(R)$,
\item[(b)] $\ehk (I + J, R/P) = \sum_{Q \in \Minh(IR/P)} \ehk(J, R/Q) \ehk(IR_Q/PR_Q)$ for all $P \in \Minh(R)$.
\end{itemize}
\end{lemma}
\begin{proof}
First, observe that if $P \in \Minh (R)$ and $Q \in \Minh(I)$ such that $P \subseteq Q$, then 
$\dim R/I \geq \dim R/(I + P) \geq \dim R/Q = \dim R/I$, so $Q \in \Minh(I + P)$ and 
the image of $Q$ in $R/P$ is in $\Minh (IR/P)$. Moreover, in this case, $\dim R/(I + P) = \dim R/I$, so 
$\Minh(I + P) \subseteq \Minh (I)$. 
And the converse is also true: if $\Minh(I + P) \subseteq \Minh (I)$ then $P$ is contained in some $Q \in \Minh(I)$.

By the additivity formula for $\ehk(IR_Q)$
\[
\smashoperator[r]{\sum_{Q \in \Minh(I)}} \ehk(J, R/Q) \ehk(IR_Q) =  
\sum_{Q \in \Minh(I)} \ehk(J, R/Q) \smashoperator[r]{\sum_{P \in \Minh(R_Q)}} \ehk(IR_Q/PR_Q) \length (R_P).
\]
If $P \in \Minh (R_Q)$, by definition, $\dim R_Q/PR_Q = \hght Q$. 
So, since $Q \in \Minh (I)$ and $\dim R/I + \hght I = \dim R$, $P \in \Minh (R)$.

Let $\Lambda = \cup \Minh (R_Q) \subseteq \Minh(R)$ where the union is taken over all $Q \in \Minh(I)$. 
In the formula above, we change the order of summations to obtain
\begin{align*}
\smashoperator[r]{\sum_{Q \in \Minh(I)}} \ehk(J, R/Q) \ehk(IR_Q)
= \sum_{P \in \Lambda} \length (R_P) \sum_{\substack{Q \in \Minh(I)\\ P \in \Minh (R_Q)}} \ehk(J, R/Q) \ehk(I R_Q/PR_Q).
\end{align*}
By the observation in the beginning of the proof, 
\[
\sum_{\substack{Q \in \Minh(I)\\ P \in \Minh (R_Q)}} \ehk(J, R/Q) \ehk(I R_Q/PR_Q) 
= \sum_{Q' \in \Minh (IR/P)} \ehk(J, R/Q') \ehk(I R_Q'/PR_Q').
\]
If the first sum is not empty (i.e., when $P \subseteq Q$ for some $Q \in \Minh(I)$), 
then $J$ is still a system of parameters modulo $I + P$ 
because it is a system of parameters modulo $Q$. 
Thus, in this case, by Lemma~\ref{HK inequality},
\[
\sum_{\substack{Q \in \Minh(I)\\ P \in \Minh (R_Q)}} \ehk(J, R/Q) \ehk(I R_Q/PR_Q) \leq \ehk (I + J, R/P).
\] 
But now, we can use the additivity formula for $I + J$, so 
\begin{align*}
\sum_{Q \in \Minh(I)} \ehk(J, R/Q) \ehk(IR_Q) &= 
\sum_{P \in \Lambda} \length (R_P) \sum_{\substack{Q \in \Minh(I)\\ P \in \Minh (R_Q)}} \ehk(J, R/Q) \ehk(I R_Q/PR_Q)
\\&\leq \sum_{P \in \Minh (R)} \length (R_P)\ehk (I + J, R/P) = \ehk(I + J),
\end{align*}
which finishes the proof.
\end{proof}

\begin{corollary}\label{equi reduced}
Let $(R, \mf m)$ be a local ring of characteristic $p > 0$ and $I$ be an ideal such that 
$\hght I + \dim R/I = \dim R$. Let $J$ be generated by a system of parameters modulo $I$.
Then 
\[
\ehk (I + J) = \sum_{Q \in \Minh(I)} \ehk(J, R/Q) \ehk(IR_Q)
\]
if and only if 
\[
\ehk (I + J, \red{R}) = \sum_{Q \in \Minh(I)} \ehk(J, R/Q) \ehk(I(\red{R})_Q).
\]
\end{corollary}
\begin{proof}
This follows from Lemma~\ref{equi mod min}, since both conditions are equivalent for $R$ and $\red{R}$. 
\end{proof}

The next lemma shows that equimultiplicity is stable under completion. Before starting the proof, we want to recall that 
Hilbert--Kunz multiplicity is well-behaved with respect to completion, namely 
$\ehk (I, M) = \ehk_{\widehat{R}}(I\widehat{R}, M \otimes_R \hat{R})$
for any finite $R$-module $M$ and $\mf m$-primary ideal $I$ in a local ring $(R, \mf m)$.

\begin{lemma}\label{equi complete}
Let $(R, \mf m)$ be a local ring of positive characteristic $p > 0$ and $I$ be an ideal such that 
$\hght I + \dim R/I = \dim R$. Let $J$ be generated by  a system of parameters modulo $I$.
Then 
\[
\ehk (I + J) = \sum_{Q \in \Minh(I)} \ehk(J, R/Q) \ehk(IR_Q)
\]
if and only if 
\[
\ehk ((I + J) \hat{R}) = \sum_{P \in \Minh(I\hat{R})} \ehk(J\hat{R}/P, \hat{R}/P) \ehk(I\hat{R}_P).
\]
\end{lemma}
\begin{proof}
First, we observe that $\ehk (I + J) = \ehk ((I + J)\hat{R})$ so we need to show that the right-hand sides are equal.

Let $Q \in \Minh(I)$. 
Because $\hat {R/Q} = \hat{R}/Q\hat{R}$,
$\ehk(J, R/Q) = \ehk(J\hat{R}/Q, \hat{R}/Q)$.
So, using the additivity formula for $\ehk(J\hat{R}/Q, \hat{R}/Q)$,
\[
\ehk(J, R/Q) = \ehk(J\hat{R}/Q, \hat{R}/Q) = 
\sum_{P \in \Minh(Q\hat{R})} \ehk(J\hat{R}/P, \hat{R}/P) \length (\hat{R}_P/Q\hat{R}_P).
\]
Since there is a flat map $R_Q \to \hat{R}_Q \to \hat{R}_P$,  
it follows that 
$\ehk(IR_Q) \length (\hat{R}_P/Q\hat{R}_P) = \ehk (I\hat{R}_P)$. 
Therefore
\[
\smashoperator[r]{\sum_{Q \in \Minh(I)}} \ehk(J, R/Q) \ehk(IR_Q) = \sum_{Q \in \Minh(I)}   
\sum_{P \in \Minh(Q\hat{R})} \ehk(J\hat{R}/P, \hat{R}/P) \ehk (I\hat{R}_P).
\]
Moreover, $\cup_Q \Minh(Q\hat{R}) = \Minh(I\hat{R})$, 
because $\dim R/I = \dim \hat{R}/I\hat{R} = \dim \hat{R}/Q\hat{R}$. Thus we obtain that 
\[
\sum_{Q \in \Minh(I)} \ehk(J, R/Q) \ehk(IR_Q) =    
\sum_{P \in \Minh(I\hat{R})} \ehk(J\hat{R}/P, \hat{R}/P) \ehk(I\hat{R}_P).
\]
\end{proof}

\begin{corollary}\label{excellent capture}
Let $(R, \mf m)$ be an excellent equidimensional local ring of characteristic $p > 0$ and let $I$ be an ideal.  
Then the following are equivalent:
\begin{enumerate}
\item $I$ satisfies colon capturing (as in Definition~\ref{capturing def}),
\item for some (equivalently, every) ideal $J$ generated by a system of parameters modulo $I$,
\[
\ehk (I + J) = \sum_{P \in \Minh(I)} \ehk (J, R/P) \ehk (I, R_P).
\]
\end{enumerate}
\end{corollary}
\begin{proof}
By Corollary~\ref{equi reduced}, both conditions are independent of the nilradical. 
Thus we can assume that $R$ is reduced, so since $R$ is excellent, by Theorem~\ref{test exist},
it has a locally stable test element. Last, since $R$ is an excellent equidimensional reduced ring, 
it is formally unmixed and we can apply Theorem~\ref{HK equimultiple}.
\end{proof}

\begin{corollary}\label{equivalence}
Let $(R, \mf m)$ be a local ring of positive characteristic $p > 0$ and $I$ be an ideal such that 
$\hght I + \dim R/I = \dim R$. 
If for some parameter ideal $J = (x_1, \ldots, x_m)$ modulo $I$,
\[
\ehk (I + J) = \sum_{P \in \Minh(I)} \ehk (J, R/P) \ehk (I, R_P),
\]
then same is true for all systems of parameters.
\end{corollary}
\begin{proof}
First, we use Lemma~\ref{equi complete} to reduce the question to the completion of $R$,
note that  $\hght I\hat{R} + \dim \hat{R}/I\hat{R} = \dim \hat{R}$. Thus we assume that $R$ is complete.

Now, condition (a) of Lemma~\ref{equi mod min} is independent of $J$. 
So, it is enough to show that the claim holds in a complete domain. 
But a complete domain has a locally stable test element by Theorem~\ref{test exist} 
and the claim follows from Proposition~\ref{colon capture}.
\end{proof}



\subsection{Localization and specialization}\label{loc spec}
It is time to study what happens to Hilbert--Kunz multiplicity after localization and specialization.
We will show that localization works under mild assumptions and establish a Bertini-type theorem for specialization.
As a corollary, we will recover condition (c) of Theorem~\ref{HS equimultiple}.

Localization of Hilbert--Kunz multiplicity is a direct consequence of the colon capturing property.

\begin{proposition}\label{equi localizes}
Let $(R, \mf m)$ be a formally unmixed local ring with a locally stable test element $c$
or an excellent locally equidimensional local ring.
Let $I$ be a Hilbert--Kunz equimultiple ideal and $\mf p$ be a prime ideal containing $I$.
Then $IR_\mf p$ is still Hilbert--Kunz equimultiple.
\end{proposition}
\begin{proof}
Let $x_1, \ldots, x_n$ be a part of a system of parameters modulo $I$ that descends to a system of parameters in
$R_\mf p$ modulo $IR_\mf p$.
By the colon capturing property in $R$ (Corollary~\ref{excellent capture}, Theorem~\ref{HK equimultiple}), 
for all $i$, $x_{i + 1}$ is not a zerodivisor modulo $(\frq{(I, x_1, \ldots, x_i)})^*R_\mf p$. 
Hence, by Lemma~\ref{tc nzd}, $x_{i+1}$ is a regular element modulo $(\frq{(I, x_1, \ldots, x_i)}R_\mf p)^*$.
Thus $IR_\mf p$ is Hilbert--Kunz  equimultiple.
\end{proof}

Now we turn our attention to specialization. 
Let us recall a slight generalization of a result of Kunz 
(\cite[Proposition~3.2]{Kunz2}, \cite[Proposition~2.13]{WatanabeYoshida}).

\begin{proposition}\label{HK specialization}
Let $(R, \mf m)$ be a local ring and let $x$ be an element of a system of parameters.
Then $\ehk(I) \leq \ehk(IR/(x))$.
\end{proposition}

\begin{proposition}\label{nice specialization}
Let $(R, \mf m)$ be a formally unmixed local ring with a locally stable test element $c$.
If $I$ is a Hilbert--Kunz equimultiple ideal, then there exists $x \in \mf m$, a parameter element modulo $I$,
such that for every parameter ideal $J$ modulo $(I,x)$
\[
\ehk(I + J, R/(x)) = \sum_{P \in \Minh (I)} \ehk(IR_P) \ehk((J,x), R/P).
\]
\end{proposition}
\begin{proof}
If $\dim R = 1$, any $P \in \Minh (I)$ is a minimal prime. 
Take a parameter element $x$ of $R$. 
Since $I$ is Hilbert--Kunz equimultiple
\[
\ehk(I, x) = \sum_{Q \in \Minh(I)} \ehk (I, R_Q) \ehk (xR/Q) = \sum_{Q \in \Min(I)} \length (R_Q)\ehk (xR/Q),
\]
where we have used that $R_Q$ is Artinian.
On the other hand, by the additivity formula
\[
\ehk(I, x) = \sum_{P \in \Minh(R)} \length (R_P) \ehk ((I,x)R/P).
\]
If $P$ contains $I$, then $\ehk((I,x)R/P) = \ehk (x, R/P)$, so after comparing the two formulas
we see that $\Minh(I) = \Minh(R)$.
Since $R$ is unmixed, $x$ is not a zero divisor, so
by the additivity formula and Proposition~\ref{hs param}
\[
\ehk (IR/(x)) = \length (R/(x)) = \eh(x) = \sum_{P \in \Minh(R)} \length (R_P) \ehk (x, R/P) = \ehk(I, x).
\]

Now, suppose that $\dim R \geq 2$. As in Corollary~\ref{excellent capture}, we may 
assume that $R$ has a locally stable test element $c$.
By prime avoidance, there exists a parameter element $x$ modulo $I$ that 
to any minimal prime of $(c)$.  Take an arbitrary parameter ideal $J$ modulo $(I, x)$.
Since $I$ is satisfies colon capturing, $x$ is regular modulo $(\frq{(I + J)})^*$ for all $q$. 
Then, by Proposition~\ref{hs param} and the additivity formula
\[
\length \left(R/((\frq{(I + J)})^*, x) \right) = \eh \left(x, R/(\frq{(I + J)})^* \right) = 
\smashoperator[l]{\sum_{Q \in \Minh (I + J)}} \ehk \left(x, R/Q \right) \length \left(R_Q/(\frq{(I + J)})^*R_Q \right).
\]

Further consider an exact sequence
\[R \xrightarrow{c} R \to R/(c) \to 0\]
and tensor it with $R/(\frq{(I + J)}, x)$. 
Since $c (\frq{(I + J)})^* \subseteq \frq{(I + J)}$, the sequence 
\[
R/(x, (\frq{(I + J)})^*) \xrightarrow{c} R/(x, \frq{(I + J)}) \to R/(c, x, \frq{( I + J)}) \to 0
\]
is also exact. 
Note that $\dim R/(x, c) = \dim R/(x) - 1$, so after passing to the limit we obtain that
\[
\ehk ((I + J)R/(x)) \leq \lim_{q \to \infty} \frac{\length \left(R/((\frq{(I + J)})^*, x)\right)}{q^{d-1}}
\]
Thus the formula above gives us that
\[
\ehk \left((I + J)R/(x) \right)
\leq \sum_{Q \in \Minh (I + J)} \ehk \left(x, R/Q \right) \ehk ((I + J)R_Q).
\]
Proposition~\ref{HK specialization} gives us the opposite inequality, so
\[
\ehk \left((I +J)R/(x) \right)
= \sum_{Q \in \Minh (I + J)} \ehk \left(x, R/Q \right) \ehk((I + J)R_Q).
\]
This gives the statement if $\dim R/I = 1$. Otherwise, we use that $IR_Q$ is still equimultiple, 
so 
\[
\ehk((I + J)R_Q) = \sum_{Q \supseteq P \in \Minh(I)} \ehk (I R_P) \ehk (J, R_Q/PR_Q). 
\]
Plugging this in the previous formula and changing the order of summation we get that
\begin{align*}
\ehk \left((I + J)R/(x) \right)
&= \sum_{P \in \Minh (I)} \ehk (I R_P) 
\sum_{Q \in \Minh (JR/P)} \ehk \left(x, R/Q \right) \ehk(J R_Q/PR_Q)
\\&= \sum_{P \in \Minh (I)} \ehk (I R_P) \ehk((x, J), R/P),
\end{align*}
where the last equality holds by Proposition~\ref{param ass formula}.
\end{proof}

As a first corollary we obtain a Bertini-type theorem for Hilbert--Kunz equimultiplicity.

\begin{corollary}\label{equi specialization}
Let $(R, \mf m)$ be a local ring which is either
formally unmixed with a locally stable test element $c$ or excellent and locally equidimensional.
If $I$ is a Hilbert--Kunz equimultiple ideal, then there exists $x \in \mf m$, a parameter element modulo $I$,
such that $IR/(x)$ is Hilbert--Kunz equimultiple.

\end{corollary}
\begin{proof}
Choose $x$ as given by Proposition~\ref{nice specialization}, so if $(J, x)$ is a parameter ideal modulo $I$,
then
\[
\ehk((I + J)R/(x)) = \sum_{P \in \Minh (I)} \ehk(IR_P) \ehk((J,x)R/P) = \ehk (I + (J, x)).
\]
By the additivity formula in Proposition~\ref{param ass formula}
\[
\ehk ((J, x)R/P) = \sum_{Q \in \Minh (P,x)} \ehk (JR/Q) \ehk (xR_Q/PR_Q),
\]
so 
\[
\ehk((I + J)R/(x)) =  \sum_{P \in \Minh (I)} \ehk(IR_P) \sum_{Q \in \Minh (P,x)} \ehk (JR/Q) \ehk (xR_Q/PR_Q).
\]
On the other hand,  Lemma~\ref{HK inequality} shows that 
\[
\ehk((I, x)R_Q)) \geq \sum_{P \in \Minh (IR_Q)} \ehk (IR_P) \ehk (xR_Q/PR_Q),
\]
so, changing the order in summations in the previous formula, we obtain that
\[
\ehk((I + J)R/(x))  \leq \sum_{Q \in \Minh (I, x)} \ehk (JR/Q) \ehk ((I,x)R_Q) \leq 
\sum_{Q \in \Minh (I, x)} \ehk (JR/Q) \ehk (IR_Q/xR_Q),
\]
where the last inequality holds by Proposition~\ref{HK specialization}.
However, the converse holds by  Lemma~\ref{HK inequality}, and the assertion follows.
\end{proof}

In the most interesting case we obtain the following statement.

\begin{corollary}\label{equi prime modulo}
Let $(R, \mf m)$ be a formally unmixed local ring with a test element $c$.
Let $\mf p$ be a prime ideal of $R$ such that $R/\mf p$ is a regular ring.
Then the following are equivalent:	
\begin{enumerate}
\item $\ehk (\mf p) = \ehk(R/(x))$ for some minimal generator $x$ of $\mf m$ modulo $\mf p$,
\item $\ehk (\mf p) = \ehk(R/(y))$ for some parameter $y$ of $\mf m$ modulo $\mf p$,
\item $\ehk(\mf p) = \ehk(\mf m)$.
\end{enumerate}
\end{corollary}
\begin{proof}
(1) $\Rightarrow$ (2) is obvious. By Proposition~\ref{HK specialization}, 
$\ehk(R) \leq \ehk(R/(y))$ for any parameter $y$ in $R$.
Thus, we always have inequalities $\ehk(\mf p) \leq \ehk(\mf m) \leq \ehk(R/(y))$, 
and (2) $\Rightarrow$ (3) follows.

For the last implication we use Corollary~\ref{equi specialization}
and observe that in the proof of Proposition~\ref{nice specialization}
$x$ can be taken to be a minimal generator of $\mf m$: an element of $\mf m$
not contained in $\mf m^2$ and the minimal primes of $c$.
 \end{proof}

Finally, let us recover a Hilbert--Kunz analogue of condition (c) in Theorem~\ref{HS equimultiple}.

\begin{corollary}\label{modulo sop}
Let $(R, \mf m)$ be a local ring which is either
formally unmixed with a test a locally stable test element $c$ or excellent and locally equidimensional.
If $I$ is an ideal then $I$ is Hilbert--Kunz equimultiple 
if and only if there is a system of parameter $J = (x_1, \ldots, x_r)$ 
modulo $I$ which is a part of a system of parameters in $R$, and such that
\begin{itemize}
\item if $r < \dim R$ then $\sum_{P \in \Minh (I)} \ehk(J, R/P) \ehk(IR_P) = \ehk(I R/J)$
\item if $r = \dim R$ then $\sum_{P \in \Minh (I)} \ehk(J, R/P) \ehk(IR_P) = \ehk (J)$,
\end{itemize}
\end{corollary}
\begin{proof}
First suppose that $I$ is Hilbert--Kunz equimultiple, so 
\[\sum_{P \in \Minh (I)} \ehk(J, R/P) \ehk(IR_P) = \ehk (I + J).\]
If $r = \dim R$, then we can use the argument in the first part of the proof of Proposition~\ref{nice specialization}
and obtain that $\Minh (I) = \Minh (R)$. Thus by the additivity formula for any system of parameters $J$
\[
\sum_{P \in \Minh (R)} \ehk(J, R/P) \ehk(IR_P) = \sum_{P \in \Minh (R)} \ehk(J, R/P) \length (R_P) = \ehk (J).
\]
In general, we build $J$ by repeatedly applying Proposition~\ref{nice specialization} in $R/(x_1, \ldots, x_i)$
until $i = r$.

For the converse, we note that a repeated use of Proposition~\ref{HK specialization} we yield the inequality
\[
\ehk(IR/J) \geq \ehk (I + J) \geq \sum_{P \in \Minh (I)} \ehk(IR_P) \ehk(JR/P).
\]
\end{proof}

%
%

\section{Applications}\label{Applications}

It is time to discuss the applications of the developed theory. 
First, we will show how equimultiplicity forces the tight closure of Frobenius powers to be unmixed.
It is important that this happens to be an equivalent condition in dimension one, and, furthermore, 
we can make it global and apply it to study Hilbert--Kunz multiplicity on the Brenner--Monsky hypersurface.

\subsection{Equimultiplicity and unmixedness of tight closure}

We start with the following consequence of our machinery.

\begin{proposition}\label{locunmix}
Let $(R, \mf m)$ be a local ring of positive characteristic $p > 0$.
Suppose that either $R$ is formally unmixed with a locally stable test element $c$
or is excellent and equidimensional.
If $I$ is an equimultiple ideal for Hilbert--Kunz multiplicity, then $\Ass (\frq{I})^* = \Minh (I)$ for all $q$. 

\end{proposition}
\begin{proof}
By Theorem~\ref{HK equimultiple} and Corollary~\ref{excellent capture}, 
we know that $I$ satisfies colon capturing.
Any prime ideal $\mf q \supset I$ such that $\mf q \notin \Minh (I)$ will contain a parameter element $x$.
Then by the colon capturing property
\[
(\frq{I})^* : x = (\frq{I})^*,
\] 
so $\mf q$ is not an associated prime.
\end{proof}

\begin{corollary}\label{prime primary}
Let $(R, \mf m)$ be a local ring of positive characteristic $p > 0$ and suppose further 
that either $R$ is formally unmixed with a locally stable test element $c$
or is excellent and equidimensional.
If $\mf p$ is a prime ideal such that $R/\mf p$ is regular and $\ehk (\mf p) = \ehk (\mf m)$,
then $(\frq{\mf p})^*$ is $\mf p$-primary for all $q$.
\end{corollary}

It is extremely useful that the converse holds if $R/\mf p$ has dimension one (Corollary~\ref{general HK equimultiplicity}).
But before we will show this importance, we need to globalize our criterion. 

\begin{lemma}\label{globprim}
Let $R$ be a ring of characteristic $p > 0$. Suppose $R$ has a test element $c$.
The following are equivalent:
\begin{enumerate}
\item[(a)] $(\frq{\mf p})^*$ are $\mf p$-primary for all $q$,
\item[(b)] $(\frq{\mf p}R_{\mf q})^*$ are $\mf p$-primary for all $q$
and all prime ideals $\mf q \supseteq \mf p$,
\item[(c)] $(\frq{\mf p}R_{\mf m})^*$ are $\mf p$-primary for all $q$
and all maximal ideals $\mf m$ that contain $\mf p$.
\end{enumerate}
\end{lemma}
\begin{proof}
The previous lemma shows $(a) \Rightarrow (b)$ and $(b)$ clearly implies $(c)$.
So, we need to show $(c) \Rightarrow (a)$.

Suppose there exists an embedded prime ideal $\mf q$. Then there is $u$ such that $\mf qu \subseteq (\frq{\mf p})^*$, but 
$u \notin (\frq{\mf p})^*$.
A result of Hochster and Huneke (\cite[Proposition~6.1]{HochsterHuneke1}) asserts that 
a tightly closed ideal is the intersection of the tightly closed ideals
containing $I$ that are primary to a maximal ideal. Therefore 
$\frq{\mf p}$ is contained in an ideal $J$ primary to some maximal ideal $\mf m$ 
and such that $u \notin J^*$.  
Since $J$ is $\mf m$-primary, \cite[Lemma~3.5]{AHH} shows that $(J)^* R_{\mf m} = (JR_\mf m)^*$, so, 
since $J^*$ is also $\mf m$-primary, $u \notin (JR_\mf m)^*$.

On the other hand,
\[
u\mf p'R_{\mf m} \subseteq (\frq{\mf p})^*R_{\mf m} \subseteq (\frq{\mf p}R_{\mf m})^* \subseteq (JR_{\mf m})^*.
\]
Thus $u \notin (\frq{\mf p}R_{\mf m})^*$ and $\mf p'R_\mf m$ consists of zero divisors on $(\frq{\mf p}R_{\mf m})^*$.
\end{proof}

Now, we can give a global version of Corollary~\ref{prime primary}.

\begin{corollary}
Let $R$ be a locally equidimensional ring and let $\mf p$ be a prime ideal such that $R/\mf p$ is regular. 
If $\ehk (\mf m) = \ehk(\mf p)$ for all maximal (equivalently, all prime) ideals $\mf m$ containing $\mf p$,
then $(\frq{\mf p})^*$ are $\mf p$-primary for all $q$.
\end{corollary}

The following result is a global version of the equimultiplicity criterion.

\begin{corollary}\label{global char}
Let $R$ be an excellent domain and let $\mf p$ be a prime ideal such that $R/\mf p$ is a regular ring of dimension one.
Then $\ehk(\mf m) = \ehk(\mf p)$ for all maximal ideals $\mf m$ containing $\mf p$ if and only 
if $(\frq{\mf p})^*$ is $\mf p$-primary for all $q$.
\end{corollary}

\begin{remark}
It seems that equimultiplicity is a strong condition.
For simplicity,  suppose $R$ is an excellent domain and $\mf p$ is a prime ideal of dimension $1$.
Then by the previous corollary, if there exists an open equimultiple subset of $\Max {R/\mf p}$,
then we can find an element $f \notin \mf p$ such that $(\frq{\mf p}R_f)^*$ is $\mf p$-primary for all $q$. 
In view of \cite[Lemma~3.5]{AHH}, this forces tight closure of all $\frq{\mf p}R_f$ to commute with localization at any multiplicatively closed set.

This suggests that equimultiplicity can be used to find further examples 
where tight closure does not commute with localization.
We are also able to use the Brenner--Monsky counter-example to study Hilbert--Kunz multiplicity. 
\end{remark}

\subsection{The Brenner--Monsky example}

It is time to show that our methods can be used to establish some interesting results in well-known example. 
First, let us introduce the Brenner--Monsky hypersurface 
\[R = F[x, y, z, t]/(z^4 +xyz^2 + (x^3 + y^3)z + tx^2y^2),\] 
where $F$ is an algebraic closure of $\mathbb Z/2\mathbb Z$.
This hypersurface parametrizes a family of quartics studied by Monsky in \cite{MonskyQP}.

Since $R$ is a quotient of a polynomial ring over an algebraically closed field, it is F-finite.
Also, $R$ is a domain, so it has a locally stable test element. 
Let $P = (x, y, z)$ then $R/P \cong F[t]$ is a regular ring and $P$ is prime.
In \cite{BrennerMonsky}, Brenner and Monsky showed that tight closure does not commute with localization at 
$S = F[t] \setminus \{0\}$.
Namely, they showed that $h(t)y^3z^3 \notin (P^{[4]})^*$ for all $h(t) \in S$, 
but the image of $y^3z^3$ is contained in $(P^{[4]}S^{-1}R)^*$. 

We want to understand the values of Hilbert--Kunz multiplicity on the maximal ideals containing $P$.
First, we will need the following result of Monsky.

\begin{theorem}\label{BM ehk}
Let $K$ be an algebraically closed field of characteristic $2$. 
For $\alpha \in K$ let $R_\alpha = K[[x,y,z]]/(z^4 +xyz^2 + (x^3 + y^3)z + \alpha x^2y^2)$.
Then
\begin{enumerate}
\item $\ehk (R_\alpha) = 3 + \frac 12$, if $\alpha = 0$,
\item $\ehk (R_\alpha) = 3 + 4^{-m}$, 
if $\alpha \neq 0$ is algebraic over $\mathbb Z/2\mathbb Z$, where 
$m = [\mathbb Z/2\mathbb Z(\lambda): \mathbb Z/2\mathbb Z]$ for $\lambda$ such that $\alpha = \lambda^2 + \lambda$
\item $\ehk (R_\alpha) = 3$ if $\alpha$ is transcendental over $\mathbb Z/2\mathbb Z$.
\end{enumerate} 
\end{theorem}
\begin{proof}
The last two cases are computed by Monsky in \cite{MonskyQP}.
For the first case, we note that in characteristic $2$ we can factor 
\[
z^4 +xyz^2 + (x^3 + y^3)z = z(x + y + z) ((x + y + z)^2 + zy).
\]
Thus by the additivity formula,
\[
\ehk(R_0) = \ehk (K[x,y]) + \ehk(K[x,y,z]/(x + y + z)) + \ehk(K[x, y, z]/((x + y + z)^2 + zy))
\]
and the claim follows.
\end{proof}

The developed machinery allows us to study Hilbert--Kunz multiplicity on the hypersurface
without computing it directly:  rather we use Monsky's computation in the specialization $R_\alpha$.

\begin{proposition}\label{HK counterexample}
Let $R = F[x, y, z, t]/(z^4 +xyz^2 + (x^3 + y^3)z + tx^2y^2)$, where $F$ is the algebraic closure of $\mathbb Z/2\mathbb Z$.
Then $\ehk (P) = 3$ for the prime ideal $P = (x, y, z)$ in $R$, 
but $\ehk(\mf m) > 3$ for any maximal ideal $\mf m$ containing $P$.
\end{proposition}
\begin{proof}
First of all, Cohen's structure theorem (\cite[p.211]{Matsumura}) shows that 
\[\widehat{R_P} \cong F(t)[[x,y,z]]/(z^4 +xyz^2 + (x^3 + y^3)z + tx^2y^2),\] 
so by the preceding theorem $\ehk(R_P) = \ehk (\widehat{R_P}) = 3$.

Second, since $F$ is algebraically closed, all maximal ideals containing $P$ are of the form 
$(P, t - \alpha)$ for $\alpha \in F$. 
By Monsky's result, $\ehk (R_{\mf m}/(t - \alpha)) > 3 = \ehk(P)$, since $\alpha$ is algebraic.
So, since $R/(t - \alpha)$ is reduced, $\ehk(\mf m) > \ehk(P)$ by Corollary~\ref{equi prime modulo}.
\end{proof}

We list two easy consequences of this result.

\begin{corollary}\label{not locally constant}
The stratum $\{\mf p \mid \ehk (\mf p) = 3\}$ is not locally closed.
In particular, the set $\{\mf q\mid \ehk(\mf q) \leq 3\}$ is not open.
\end{corollary}
\begin{proof}
If it was an intersection of a closed set $V$ and an open set $U$, then
the intersection $V(P) \cap U$ should be non-empty and open in $V(P)$. 
In particular, it would contain infinitely many maximal ideals $\mf m$ containing $P$.
\end{proof}

\begin{corollary}\label{inf many values}
The set $\{\ehk(\mf m) \mid P \subset \mf m\}$ is infinite.
\end{corollary}
\begin{proof}
In the notation of the proof of Proposition~\ref{HK counterexample}, 
we have that $\ehk(\mf mR/(t- \alpha)) \geq \ehk (\mf m) > \ehk(P)$.
Now the claim follows from Theorem~\ref{BM ehk}, since $\ehk (\mf m)$ tends to $3$ 
when $[\mathbb Z/2\mathbb Z(\lambda): \mathbb Z/2\mathbb Z]$ grows.
\end{proof}

Another application of our methods is a quick calculation of the associated primes of $\frq{P}$. 
Using the calculations that Monsky made to obtain Theorem~\ref{BM ehk}, 
Dinh (\cite{Dinh}) proved that $\bigcup_q \Ass (\frq{P})$ is infinite. 
However, he was only able to show that the maximal ideals corresponding to the
irreducible factors of $1 + t + t^2 + \ldots + t^q$ appear as associated primes,
while our methods give all associated primes of the Frobenius powers and their tight closures.

\begin{proposition} \label{infinite ass}
In the Brenner--Monsky example, 
\[\bigcup_q \Ass (\frq{P})^* = \bigcup_q \Ass (\frq{P}) = \Spec {R/P}.\] 
In particular, the set is infinite.
\end{proposition}
\begin{proof}
Clearly, $P$ is an associated prime, so we only need to check the maximal ideals.

First, we prove that any prime $\mf m$ that contains $P$ is associated to some $(\frq{P})^*$.
If not,  then $(\frq{P})^*R_\mf m$ are $P$-primary for all $q$.
Note that $(\frq{P})^*R_\mf m \subseteq (\frq{P}R_\mf m)^*$, thus by Corollary~\ref{tc primarity}
$(\frq{P}R_\mf m)^*$ is $P$-primary for any $q$. Therefore, by Corollary~\ref{general HK equimultiplicity},
$\ehk(\mf m) = \ehk(P)$, a contradiction.

For the second claim, let $\mf m$ be any maximal ideal containing $P$. 
Since $\ehk (P) < \ehk(\mf m)$ and $R/P$ is regular, $\mf m$ is an associated prime of $(\frq{P}R_\mf m)^*$ for some $q$.
Thus, there exists $u \notin (\frq{P}R_\mf m)^*$ such that
\[
cu^{q'} \mf m^{[q']} \subseteq P^{[qq']}R_\mf m.
\]
Now, if $\mf m$ is not an associated prime of any $P^{[qq']}$, then we would have $u \in (\frq{P}R_\mf m)^*$, a contradiction.
\end{proof}

\begin{remark}
The presented example shows that Hilbert--Kunz multiplicity need not be locally constant 
when tight closure does not commute with localization. 
However, it is not clear whether it should be locally constant if, in a particular example, tight closure commutes with localization.
Even in this case $\bigcup_q \Ass (\frq{\mf p})^*$ might be infinite (\cite{SinghSwanson}), 
and it is not clear why the intersection of the embedded primes must be greater than $\mf p$.
\end{remark}

\section{Open problems}

There is a number of question laying further in the direction of this work.

We can fill the missing part (a) of Theorem~\ref{HK intro}, by 
defining the analytic spread for tight closure using part (d) of the theorem, 
as the maximal $r$ such that there exists 
parameters $x_1, \ldots, x_r$ modulo $I$ such that
\[
x_{i + 1} \text { is regular modulo } (\frq{(I, x_1, \ldots, x_i)})^* \text { for all $q$ and $i$.}
\]
This definition is analogous to a definition of analytic spread, 
but we should hope that there is a more conceptual way to define it.

\begin{question}\label{Spread}
Is there an intrinsic definition of the analytic spread for tight closure
that will finish the equivalence of  Theorems~\ref{HS equimultiple} and \ref{HK intro}?
\end{question}

Neil Epsten and Adela Vraciu have developed a theory of *-spread (\cite{Epstein, EpsteinVraciu}).
Epstein defined $\ell^* (I)$  as the minimal 
number of elements needed to generate $I$ up to tight closure. However,  \cite[Lemma~1]{EpsteinVraciu}
shows that if $I$ is an ideal and $x$ is a parameter modulo $I$
there exists an integer $n$ such that $\ell^* (I, x^n) = \ell^* (I) + 1$, but $x$ does not have 
to be a regular element modulo $(\frq{I})^*$.
Perhaps, one should rather think about the analytic spread as the dimension of the fiber cone?

\begin{question}
What can be said about the maximum value locus of the Hilbert--Kunz multiplicity? 
Are the irreducible components regular?
\end{question}

In the Brenner--Monsky example the locus consists of finitely many maximal ideals, so 
the irreducible components are indeed smooth. The author does not know any computation 
of Hilbert--Kunz multiplicity in a two-parameter family, so there is no real example to test this question on.

\begin{question}
What is the behavior of Hilbert--Kunz multiplicity under a blow-up?
If the previous question has a positive answer, we may hope that one can always 
choose a component such that its blow-up will not increase the Hilbert--Kunz multiplicity. 
\end{question}

\begin{question}
This work suggests that Hilbert--Kunz multiplicity might be ``too sensitive''. Is there a suitable coarsening?
\end{question}

Perhaps, Hilbert--Kunz multiplicity should be paired with a less sensitive invariant, such as the Hilbert--Samuel function,
and should be only used if that invariant does not provide much information.
Another possible hope is to ``round'' Hilbert--Kunz multiplicity, perhaps the fact that Hilbert--Kunz multiplicity 
cannot be too close to $1$ is a sign?

It is widely believed that Hilbert--Kunz multiplicity can be defined in characteristic zero
by taking the limit of the Hilbert--Kunz multiplicities of the reductions mod $p$ (see \cite{BrennerLiMiller, Trivedi} for partial results). 

\begin{question}
What are the geometric properties of the characteristic zero Hilbert--Kunz multiplicity?
Can it be used to give a new proof of Hironaka's theorem on the resolution of singularities?

\end{question}

\specialsection*{Acknowledgements}
The results of this paper are a part of the author's thesis written under Craig Huneke in the University of Virginia. 
I am indebted to Craig for his support and guidance. This project would not be possible without his constant encouragement.

I also want to thank 
Mel Hochster and the anonymous referee who carefully read and helped to improve this manuscript.

\bibliographystyle{plain}
\bibliography{equibib}

\end{document}